\documentclass[12pt]{article}  % default square logo 
\usepackage{amssymb}

\usepackage[dvipsnames]{xcolor}					% provides colors in document
\usepackage{color}								% allows coloring
\usepackage[toc,title]{appendix}
\usepackage{comment}							% provides writing comments	
\usepackage[numbers]{natbib}					% enhanced citing options
\usepackage{pdfpages}                           % allows integrating pdf files
\usepackage{xspace}								% define commands that appear not to eat spaces
\usepackage{epstopdf} 							% converts EPS to PDF using Ghostscript
\usepackage{nomencl}							% enable Nomenclature
\usepackage{etoolbox}							% enable groupgings in Nomencaltura
\usepackage{longtable}							% tables spanning pages
\usepackage{authblk} 
\usepackage{rotating}							% enable vertical pages and tables
\usepackage{float}								%better table placement

\usepackage[ margin=2cm]{geometry}
\usepackage{bbm}

\renewcommand\nomgroup[1]{%
	\item[\bfseries
	\ifstrequal{#1}{M}{Matrices}{%
		\ifstrequal{#1}{V}{Vectors}{%
			\ifstrequal{#1}{Sets}{Sets}{}}}%
	]}

\usepackage{graphicx}							% graphic environment
\usepackage{multirow}							% multiple rows in lists
\usepackage{paralist}							% enhanced lists
\usepackage{verbatim}							% discards latex commands and shows signs as the are (e.g. easily write source codes)
\usepackage{longtable}							% allows tables over several pages
\usepackage{booktabs}							% controls lines in tables
\usepackage{subcaption}							% combines two figures as float object

\makeatletter
\def\BState{\State\hskip-\ALG@thistlm}
\makeatother

\usepackage{tikz}								% create figures
\usetikzlibrary{graphs,graphs.standard,quotes}	% graph environments
\usetikzlibrary{backgrounds}					% include backgrounds				
\usetikzlibrary{fit} 							% do not know
\usepackage[T1]{fontenc}
\usepackage{lmodern}
\usepackage[utf8]{inputenc}						% umlaute
\usepackage{soul}								% strike out words
\usepackage[normalem]{ulem}						% underline or strike out words
\usepackage{mathrsfs}							% scripture-like letters
\usepackage{dsfont}

\usepackage[ruled,vlined]{algorithm2e} 
\usepackage{setspace}
\DontPrintSemicolon
\SetKw{KwAnd}{and}
\SetKw{KwOr}{or}
\SetKw{KwTrue}{true}
\SetKw{KwFalse}{false}

\usepackage{forloop}							% provides the \forloop for writing makros
\usepackage{ifthen}								% provides the if-then loop

\colorlet{darkred}{red!80!black}
\colorlet{darkgreen}{green!60!black}
\colorlet{darkblue}{blue!80!black}
\colorlet{darkorange}{orange!70!black}
\definecolor{Green}{cmyk}{1, 0.2, 0.4, 0.1}
\definecolor{Orange}{cmyk}{0, 0.61, 0.87, 0}
\definecolor{Purple}{rgb}{0.75, 0.0, 1.0}
\definecolor{purple}{rgb}{0.5,0,0.5}
\definecolor{dgreen}{rgb}{0.1,0.9,0.5}
\definecolor{gabysgreen}{cmyk}{0.80, 0.1, 0.90, 0}

\usepackage{amssymb}							% enhanced math
\usepackage[fleqn]{amsmath}						% enhanced math
\usepackage{amsthm}								% math theorems
\usepackage{bm}									% bold math
\usepackage{sansmath}
\usepackage{mathtools}							% allows at least definition sign
\usepackage{nicefrac}							% allows 1/p
\usepackage{siunitx}							% si-units for writing scientific numbers
\usepackage{xfrac}								% for nicer fractions 

\usepackage{hyperref}							% referencing with hyperlinks
\hypersetup{
	colorlinks,
	linkcolor={blue!50!black},
	citecolor={green!50!black},
	urlcolor={blue!80!black}
}
\usepackage[capitalise,noabbrev]{cleveref}		% clever referencing (e.g., automatically adds name of reference)

\newcommand{\irg}[2]{[#1\!:\!#2]}

\newcommand{\R}{\mathbb{R}}						%real numbers
\newcommand{\N}{\mathbb{N}}						%natural numbers
\newcommand{\CPP}{\mathcal{CPP}}				%completely positive matrices
\newcommand{\x}{\vc{x}}							%bold vector x
\renewcommand{\u}{\vc{u}}						%bold vector u
\newcommand{\V}{\mathsf{V}}						%perturbation matrix
\newcommand{\lrbr}[1]{\left\lbrace #1 \right\rbrace} %Left and Right braces

\newcommand\ga{\alpha}

\newcommand\gb{\beta}

\newcommand\gd{\delta}

\newcommand\gc{\gamma}

\renewcommand\ge{\epsilon}

\newcommand\gl{\lambda}
\newcommand\ggl{\bm{\lambda}}

\newcommand\mmu{\bm{\mu}}

\newcommand\gp{\eta}

\def\x{\mathbf x}
\def\a{\mathbf a}
\def\y{\mathbf y}
\def\z{\mathbf z}
\def\s{\mathbf s}
\def\v{\mathbf v}
\def\w{\mathbf w}
\def\p{\mathbf p}
\def\f{\mathbf f}
\def\oo{\mathbf o}
\def\e{\mathbf e}
\def\u{\mathbf u}
\def\b{\mathbf b}
\def\cc{\mathbf c}

\def\g{\mathbf g}
\def\q{{\mathbf q}}

\def\AA{{\mathcal A}}
\newcommand\BB{{\mathcal B}}
\newcommand\CC{{\mathcal C}}

\newcommand\FF{{\mathcal F}}
\newcommand\GG{{\mathcal G}}

\newcommand\II{{\mathcal I}}

\newcommand\KK{{\mathcal K}}

\newcommand\MM{{\mathcal M}}

\newcommand\PP{{\mathcal P}}

\def\SS{{\mathcal S}}
\newcommand\TT{{\mathcal T}}
\newcommand\UU{{\mathcal U}}
\newcommand\VV{{\mathcal V}}

\newcommand\XX{{\mathcal X}}

\newcommand\Ab{{\mathsf{A}}}
\newcommand\Bb{{\mathsf B}}
\newcommand\Cb{{\mathsf C}}

\newcommand\Eb{{\mathsf E}}

\newcommand\Gb{{\mathsf G}}
\newcommand\Hb{{\mathsf H}}
\newcommand\Ib{{\mathsf I}}

\newcommand\Mb{{\mathsf M}}

\newcommand\Ob{{\mathsf O}}

\newcommand\Qb{{\mathsf Q}}

\newcommand\Ub{{\mathsf U}}
\newcommand\Vb{{\mathsf V}}
\newcommand\Wb{{\mathsf W}}
\newcommand\Xb{{\mathsf X}}
\newcommand\Yb{{\mathsf Y}}
\newcommand\Zb{{\mathsf Z}}

\newcommand{\tr}[1]{\mathrm{Tr}(#1)}								% trace
\newcommand{\T}{^\mathsf{T}}	
\DeclareMathOperator{\conv}{conv}									% convex hull
\DeclareMathOperator{\diag}{diag}									% diagonalize matrix to vector
\DeclareMathOperator{\Diag}{Diag}									% diganalize vector to matrix
\DeclareMathOperator{\sign}{sign}									% sign
\def\beq#1{$$}														% begin equation
\DeclareMathOperator{\interior}{int} 								% interior
\DeclareMathOperator{\cl}{cl} 									    % closure
\DeclareMathOperator{\dom}{dom}										% effective domain
\def\ol#1{{\overline #1}}

\def\eps{\varepsilon}

\def\ol#1{{\overline #1}}

\def\bea#1{\begin{array}{#1}}
	\def\ea{\end{array}}
\def\ignore#1{}

\theoremstyle{plain}
\newtheorem{thm}{Theorem} % [section]
\newtheorem{exmp}{Example} %[section]
\newtheorem{asm}{Assumption}

\newtheorem{lem}[thm]{Lemma}
\newtheorem{prop}[thm]{Proposition}
\newtheorem{cor}[thm]{Corollary} %[section]
\newtheorem*{rem}{Remark}
\theoremstyle{definition}
\newtheorem{defn}{Definition} %[section]
\theoremstyle{remark}

\usepackage{fancyhdr}

\begin{document}
	\renewcommand\Authands{ and }
	%TITLE TR
	\title{Concave tents:\\ a new tool for constructing concave reformulations\\ of a large class of nonconvex optimization problems}
	
	\author{
		Markus Gabl 
		\footnote{Faculty of Mathematics, University of Vienna, Austria; email: {\tt markus.gabl@univie.ac.at}}} %\thanks{markus.gabl@univie.ac.at}}
	% %
	%\affil{University of Vienna \Authands  }
	\maketitle
	%\tableofcontents 
	\begin{abstract}
		Optimizing a nonlinear function over nonconvex sets is challenging since solving convex relaxations may lead to substantial relaxation gaps and infeasible solutions that must be "rounded" to feasible ones, often with uncontrollable losses in objective function performance. For this reason, these convex hulls are especially useful if the objective function is linear or even concave, since concave optimization is invariant to taking the convex hull of the feasible set. Motivated by this observation, we propose the notion of concave tents, which are concave approximations of the original objective function that agree with this objective function on the feasible set, and allow for concave reformulations of the problem. Concave tents, therefore, are special cases of concave extensions, which were originally introduced by Crama 1993. In this text, we derive such concave tents for a large class of objective functions as the optimal value functions of conic optimization problems. Hence, evaluating our concave tents requires solving a conic problem. Interestingly, we can find supergradients by solving the conic dual problem, so that differentiation is of the same complexity as evaluation. 
		For feasible sets that are contained in the extreme points of their convex hull, we construct these concave tents in the original space of variables. For general feasible sets, we propose a double lifting strategy, where the original optimization problem is lifted into a higher-dimensional space in which the concave tent can be constructed with a similar effort. We investigate the relation of the so-constructed concave tents to concave envelopes and naive concave tents based on concave quadratic updates. Based on these ideas, we propose a primal heuristic for a class of robust discrete quadratic optimization problems that can be used instead of classical rounding techniques. Numerical experiments suggest that our techniques can be beneficial as an upper-bounding procedure in a branch-and-bound solution scheme. 
		
	\end{abstract}
	{\bf Keywords:} Nonconvex Optimization, Conic Optimization, Copositive Optimization, Global Optimization 
	\newpage
	
	\pagebreak	
	\section{Introduction}
	In this paper, we consider the optimization problem 
	\begin{align}\label{eqn:GeneralConvexOverNonconvex}
		\inf_{\x\in\R^n} \lrbr{f(\x) \colon \x\in\XX},
	\end{align}
	where $f \colon \R^n\rightarrow \R\cup\lrbr{\infty,-\infty}$  is a nonlinear function and $\XX\subseteq \R^n$ is a compact set, and both may be nonconvex. This is an important yet challenging problem class. Usually, it is tackled via global optimization techniques, that vary depending on the structure of $\XX$ and $f$. If $f$ were linear, one could solve the problem as a convex optimization problem if the convex hull of $\XX$, which is also a compact set denoted by $\conv(\XX)$, can be handled in a tractable manner. For example, if $\XX\subseteq\lrbr{0,1}^n$, then the convex hull has a tractable polyhedral description in many interesting cases and there are many techniques to approximate it algorithmically.  However, already in the nonlinear convex case, and even if a full description of the convex hull of $\XX$ is available, the resulting convex relaxation can in general only deliver a lower bound, since a convex nonlinear function may attain solutions at points in $\conv(\XX) \setminus \XX$. In addition to the problem of potentially weak bounds, projecting the optimal solution of the relaxation onto $\XX$, for example via rounding,  may introduce a loss in performance that can be hard to control. 
	
	However, in this text, we want to go another route, where we replace $f$ with a function $g$ that is concave on $\conv(\XX)$ but agrees with $f$ on $\XX$, which is called a \textit{concave tent} (see \cite{crama_concave_1993,tawarmalani_convex_2002}). Concave functions attain their optimum over compact convex sets at extreme points of these sets, so that for a general concave function $\hat{g}$ that overestimates $f$ on $\XX$ we have that 
	\begin{align*}
		\inf_{\x\in\R^n} \lrbr{f(\x) \colon \x\in\XX} \leq  \inf_{\x\in\R^n} \lrbr{\hat{g}(\x) \colon \x\in\XX}= 
		\inf_{\x\in\R^n} \lrbr{\hat{g}(\x) \colon \x\in\conv(\XX)},
	\end{align*}
	and we will show that the mere existence of a concave tent $g$ of $f$ over $\XX$ guarantees that the gap can be closed for $\hat{g}=g$.  Hence, we reformulate nonconvex minimization over a nonconvex set into concave minimization over a convex set. Both, of course, can be hard. Nonetheless, the transformation allows for a different set of solution strategies to be applied that can take advantage of optimization techniques over $\conv(\XX)$ and certain properties of $g$. We will particularly be interested in using these reformulations for the sake of producing upper bounds, which are very useful in global optimization procedures since they often allow for the elimination of potential solution candidates. Obtaining such a concave tent seems elusive at first, but we will show that it can be constructed for a large class of choices of $f$ and $\XX$ by exploiting conic optimization theory and convex geometry.  
	
	Concave tents are special cases of so-called \textit{concave extensions}. These are concave overestimators of $f$ which are exact on pre-specified subsets of a convex domain. They were already introduced in \cite{crama_concave_1993,tawarmalani_convex_2002} and in \cref{sec:Definitions existence and implication for optimization} will discuss more closely how concave tents relate to concave extensions. However, our exposition offers a new perspective on this topic that differs significantly from earlier treatments in \cite{crama_concave_1993,tawarmalani_convex_2002}. Indeed, their approach sought to generate relaxations and, therefore, lower bounds of \eqref{eqn:GeneralConvexOverNonconvex}, by replacing $f$ with a convex function that is identical to $f$ on certain subsets of its domain. Hence, \cite{tawarmalani_convex_2002} spoke of convex (rather than concave) extensions (i.e.\ convex underestimators that are exact on a pre-specified subset of a convex domain), while \cite{crama_concave_1993} dealt with concave extensions but considered maximization problems. Their idea was, that a convex underestimator that is exact on certain points of the domain should yield strong relaxations by design. To the contrary, we build concave reformulations in an effort to construct upper bounds, harnessing the useful properties of concave functions that let us exploit knowledge on $\conv(\XX)$.   
	
	This shift in perspective opens up a new set of techniques to approach (\ref{eqn:GeneralConvexOverNonconvex}), namely concave optimization, i.e.\ minimizing a concave function over a convex set. This class of problems is NP-hard, as it contains NP-hard special cases (see e.g.\ \cite{garey_simplified_1976,ivanescu_network_1965,mangasarian_variable-complexity_1986}). The fact that concave problems attain their optimum at an extreme point of their feasible set, a considerable advantage, motivated a plethora of literature that sought to reformulate nonconvex optimization problems into concave ones. An early approach by \cite{raghavachari_connections_1969} concerns zero-one integer programming, where the binarity constraint on the decision $\x$ can be relaxed to the unit box constraint after the objective is augmented by a nonnegative concave quadratic term which vanishes at the extreme points of the unit box, hence preserving the global optimum. In \cite{pardalos_constrained_2006}, the authors summarize similar approaches for an array of problems first introduced in \cite{falk_linear_1973, frieze_bilinear_1974,konno_maximization_1976,mangasarian_characterization_1978,thieu_relationship_1980}. In \cite{mangasarian_machine_1996} the author used concave reformulations  to tackle misclassification minimization and feature selection. A similar approach was proposed by \cite{rinaldi_concave_2010}, where concave reformulations of zero-norm optimization problems over polyhedral sets were provided and then solved via a Frank-Wolfe-type algorithm.  A  more recent approach involving concave reformulations is given by \cite{di_lorenzo_concave_2012}, tackling a category of portfolio optimization with zero norm regularization. 
	
	An important consideration that merits the shift to a concave optimization perspective is the elevated role of convex hulls of nonconvex sets. These feature prominently in the global optimization and, specifically, mixed integer linear optimization literature, see e.g. \cite{pardalos_constrained_2006,papadimitriou_combinatorial_1998}). Convex hulls of quadratically constrained feasible sets are studied in \cite{burer_how_2017} and references therein. While convex hulls and approximations thereof are usually used to generate lower bounds, we will show that concave tents allow us to exploit this machinery for the sake of computing upper bounds, which are also important in global optimization schemes.     
	
	Approaches involving convex hulls are regularly used in conjunction with linearization techniques. One of the most popular instances of such machinery is given by copositive reformulations of nonconvex quadratically constrained quadratic problems. These rest on the identity 
	\begin{align}\label{thm:chlifting}	
		\inf_{\x\in\R^n}\lrbr{\x\T \Qb\x +2 \q\T\x\colon \x\in \XX} =			\inf_{\x\in\R^n,\, \Xb\in\SS^n}\lrbr{\tr{\Qb\T\Xb}+ 2 \q\T\x \colon 
			\begin{bmatrix}
				1 & \x\T \\ \x & \Xb
			\end{bmatrix}
			\in \GG(\XX)
		},
	\end{align}
	where $\SS^n$ is the set of $n\times n$ symmetric matrices and 
	\begin{align*}
		\GG(\XX)\coloneqq \mathrm{clconv} \left\lbrace
		\left [ \bea{c}
		1 \\ \x  
		\ea \right ]
		\left [ \bea{c}
		1 \\ \x  
		\ea \right ]\T
		: \x\in \XX \right\rbrace\subseteq \SS^{n+1} \,, 
	\end{align*}
	which is proved for example in \cite[Theorem 1]{bomze_optimization_2023}. The approach is based on a lifting where the terms $x_ix_j$ are replaced by $(\Xb)_{ij}$ (or equivalently, $\x\x\T$ is replaced by $\Xb$). Since the lifted variables are constrained to the set $\GG(\XX)$, whose extreme points are dyadic matrices (i.e. matrices of the form $\x\x\T$), there is an optimal solution to the convex problem that can be decomposed into a feasible solution of the original problem, which by the equality $\tr{\Qb\Xb} = \tr{\Qb\x\x\T} = \x\T\Qb\x$ yields the same optimal value so that the relaxation gap is nullified. The key difficulty of this approach is the characterization of $\GG(\XX)$, which is only known for special instances of $\XX$, see for example \cite{anstreicher_computable_2010,burer_copositive_2009,burer_copositive_2012,burer_gentle_2015,eichfelder_set-semidefinite_2013,burer_second-order-cone_2013,sturm_cones_2003,yang_quadratic_2018}. These results have been applied, for example, to fractional optimization \cite{amaral_nonconvex_2019,he_convexification_2024} and stochastic 0-1 optimization \cite{natarajan_mixed_2011}.

	Recently, $\GG(\XX)$ has also been used to generate underestimators of optimal value functions of quadratically constrained quadratic problems in \cite{gabl_finding_2024}, which presents a continuation of ideas earlier proposed in \cite{yildirim_alternative_2022,anstreicher_convex_2012}. We want to highlight at this point, that this set will play a central role in our exposition as well. We will several times invoke \cite[Proposition 4]{gabl_finding_2024}, which states that if $\XX$ is compact, then $\GG(\XX)$ is compact even if the "$\cl$" operator were omitted from its definition so that in that case every element in $\GG(\XX)$ is a convex combination of matrices of the form  $\y\y\T\colon \y\in\lrbr{1}\times\XX$, and not merely a clusterpoint of such convex combinations.
	
	\section{Contribution}
	This text aims at achieving the following contributions: 
	\begin{itemize}
		\item We formally introduce the notion of concave tents and discuss some necessary and sufficient conditions for their existence, as well as their relations to general concave extensions. We show that nonconvex optimization over a nonconvex set is in a certain sense equivalent to optimizing the respective concave tent, if it exists, over the convex hull of that set. 
		
		\item We will construct concave tents of functions of a certain class of sets that are subsets of the extreme points of their convex hull. These concave tents will be characterized as the optimal value functions of conic optimization problems related to copositive optimization theory.  
		
		\item For the case of general feasible sets we employ a double lifting strategy, which can be interpreted as a generalization of the traditional copositive optimization paradigm. 
		
		\item We discuss the quality of our concave tents in comparison to the concave envelope and a classical construction of concave tents based on concave quadratic updates, and show that our concave tents are sandwiched between the two and that the lower bound that is the concave envelope can be attained. 
		
		\item We will discuss obtaining $\epsilon$-supergradients of our concave tents based on conic duality theory. 
		
		\item We then use our theory to construct a primal heuristic for a class of robust quadratic discrete optimization problems. Numerical experiments demonstrate that this heuristic can be useful in a branch and bound (B\&B) solution scheme. 
		
	\end{itemize}
	
	\subsection{Notation}
	Throughout the paper, matrices are denoted with sans-serif capital letters, e.g.,~$\Eb$ is the matrix of all ones, $\Ib$ is the identity matrix, and $\Ob$ the matrix of all zeros where the order will be clear from the context. Vectors will be given as boldface lowercase letters, for instance, the vector of all ones (a column of $\Eb$) is $\e$, the vector of zeros is $\oo$ and the vector $\e_i$ is the $i$-th column of $\Ib$. We will also use lowercase letters of the same symbol with indices to refer to entries of a vector denoted by that symbol, for example, $a_i = (\a)_i$ for $\a\in\R^n$. If a scalar takes the role usually assigned to a matrix, e.g.\ in a special case where an $n\times n$ matrix becomes a scalar because $n=1$, we will use italic capital letters for that matrix. For two matrices $\Ab$ and $\Bb$ we will use the binary function $\Ab\bullet\Bb \coloneqq \tr{\Ab\T\Bb}$ as the inner product known as the Frobenius product. For a symmetric matrix $\Xb$ the function $\diag(\Xb)$ maps to the vector of diagonal entries of $\Xb$.  
	
	Sets will mostly be referred to using calligraphic letters. For example, $\SS^n$ is the space of symmetric $n\times n$ matrices, $\SS_+^n$ will be the cone of positive semidefinite matrices, its interior, the cone of positive definite matrices is denoted $\SS_{++}^n$. Also, Euclidean balls of radius $r$ centered at $\cc$ will be denoted $\BB_r(\cc)\coloneqq \lrbr{\x\in\R^n\colon \|\x-\cc\|_2\leq r}\subseteq\R^n$. Exceptions to this rule are the $n$-dimensional Euclidean space $\R^n$, its nonnegative orthant $\R^n_+$, or the index set $\irg{i}{j}= \lrbr{i,i+1,\dots,j-1,j}$, where $i<j$ are integer numbers and finally the standard simplex $\Delta_n\coloneqq \lrbr{\x\in\R^n_+\colon \e\T\x=1}$. For a set $\AA$ we denote $\cl(\AA),\interior(\AA),\conv(\AA)$ and $\mathrm{cone}(\AA)$ its closure, interior, convex hull, and the union of its conic hull and the origin, respectively. For a convex function $f$ the set $\dom(f)\coloneqq \lrbr{\x\in\R^n\colon f(\x)<\infty}$ is its effective domain and the function $\cl(f)$ is its closure, i.e.\ its lower semicontinuous envelope. If $f$ is a concave function we define $\dom(f) \coloneqq \dom(-f)$,  and $\cl(f) = -\cl(-f)$. A convex function is proper if it is bigger than $-\infty$ everywhere and smaller than $\infty$ in at least one point. A concave function $f$ is proper if $-f$ is a proper convex function. We say a convex (or concave) function $f$ is closed if $f = \cl(f)$ and $f$ is proper or equal to $-\infty$ ($\infty$) if it is improper. We will also use the shorthand $(\x,\Xb)$ to indicate the matrix
	\begin{align*}
		\left [ \bea{cc}
		1 & \x\T \\ \x & \Xb
		\ea\right ]\in\SS_+^n.
	\end{align*}

	\section{Definition, existence, and implication for optimization}\label{sec:Definitions existence and implication for optimization}
	We start by formally defining what we mean by a concave tent. 
	\begin{defn}\label{def:ConcaveTent}
		Let $\XX\subseteq\R^n$ and let $f\colon \R^n\rightarrow \R$ be a function. Then a concave tent of $f$ over $\XX$ is any function $g\colon \R^n\rightarrow \R $ such that 
		\begin{itemize}
			\item[a)] $g$ is concave over $\conv(\XX)$, and
			\item[b)] $f(\x) = g(\x)$ for all $\x\in\XX$. 
		\end{itemize}
	\end{defn}
	
	The above definition is a specialization of the definition of so-called concave extensions, originally introduced in \cite{crama_concave_1993,tawarmalani_convex_2002}, which are defined as follows. Let $\CC\subseteq\R^n$ be a convex set and let $\XX\subseteq \CC$. A concave extension of a function $f\colon \XX\rightarrow\R$ over $\CC$ is any concave function $g\colon\CC\rightarrow\R$ such that $g(\x) = f(\x),\ \forall \x\in\XX$ (see \cite[Definition 1.]{tawarmalani_convex_2002}). The definition of a concave tent can be recovered by requiring $\CC = \conv(\XX)$. In addition, we only require $g$ to be concave on $\conv(\XX)$, but that is a nuance of little consequence. Similarly, our nomenclature differs slightly as we speak of concave tents over $\XX$, rather than over $\CC = \conv(\XX)$. However, this seemingly minor specialization is important since concave tents allow us to construct concave reformulations of \eqref{eqn:GeneralConvexOverNonconvex}. This represents a significant departure from the originally intended use of concave extensions, whose convex counterparts were introduced to build convex relaxations, where the idea was to have a relaxation that is exact at least in some parts of the domain. Nonetheless, many theoretical aspects of concave extensions, mainly explored in \cite{tawarmalani_convex_2002}, can be transferred to the realm of concave tents. Still, in the remainder of this section, we will provide a self-contained discussion of some of those aspects that is geared towards preparing the reader for the rest of the present text. 
	
	Firstly, note that concave tents are not necessarily concave functions, as they can be nonconcave outside of $\conv(\XX)$, and that they overestimate $f$ on $\XX$ since there they are equal, but not necessarily over $\conv(\XX)$ unless $f$ is convex.
	Now, we want to convince the reader that the above concept is meaningfully distinct from general concave overestimators as well as from the concave envelope, i.e. the smallest concave overestimator of that function over $\XX$. Indeed, an arbitrary concave overestimator may fail to meet the requirement b) in \cref{def:ConcaveTent}. Take as an example the function that is $\infty$ everywhere, which is concave over any set $\XX$ but differs from any function $f$ finite on that set. On the other hand, consider the concave envelope of $f(\x) = x^2$ over the set $\XX  = \lrbr{-1,1}$, which is just the function that is a constant equal to one within the interval $[0,1]$ and $-\infty$ everywhere else. It is a concave tent by \cref{def:ConcaveTent}, but that is also the case for the function that we obtain by adding to this concave envelope the function $1-x^2$, which also demonstrates that concave tents are not unique. However, we still can show the following simple relation between concave tents and concave envelopes. 
	
	\begin{prop}
		If a concave tent of a function $f$ over a set $\XX$ exists, then the concave envelope of $f$ over $\XX$ is a concave tent. 
	\end{prop}
	\begin{proof}
		Let $h$ be the hypothesized concave envelope and $g$ be an arbitrary concave tent, then for any $\x\in\XX$ we have $f(\x)\leq h(\x)\leq g(\x) =f(\x)$, since $h$ is the smallest concave overestimator of $f$ on $\XX$, hence also smaller than $g$ on that domain. Thus, $h$ is a concave function that agrees with $f$ over $\XX$, i.e. it is a concave tent. 
	\end{proof}

	It is, however, easy to see that a concave tent does not necessarily exist. For example, consider $f(x) = \max\lrbr{-x,x-1}, \ \XX \coloneqq\lrbr{0,1/2,1}$. We see that a concave tent is impossible since there is a point $x = 1/2$ so that any function $g$ that agrees with $f$ on $\XX$ fulfills $ -0.5 = g(1/2) = g((1/2) 0 +(1/2) 1)<1/2 g(0) +1/2 g(1) = 0$ so that $g$ cannot be concave over $\conv(\XX)$. 
	
	Later in the text, we will derive concave tents for various configurations of $\XX$ and $f$ so that existence is guaranteed in interesting cases. It is nonetheless fruitful to examine on a more fundamental level what the necessary and sufficient conditions for its existence are. Our example in the previous paragraph illustrates the critical necessary condition for the existence of a concave tent discussed in the result below:
	\begin{prop}\label{thm:ExistenceofCT}
		Let $f$ and $\XX$ be as in \cref{def:ConcaveTent}, and consider the following condition
		\begin{align}\label{cond:Existence}
			\XX\ni\x = \sum_{i=1}^{k}\gl_i\x_i,\ \ggl\in\Delta_k, \ \x_i\in\XX, \ i \in\irg{1}{k}
			\implies f(\x)\geq \sum_{i=1}^{k}\gl_i f(\x_i). 
		\end{align}
		Then, \eqref{cond:Existence} is necessary for the existence of a concave tent of $f$ over $\XX$. If $\XX$ is compact and $f$ is continuous over $\XX$, then \eqref{cond:Existence} is also sufficient.
	\end{prop} 
	\begin{proof}
		To prove necessity, assume there was such an $\x\in\XX$ that is a convex combination of $\x_i\in\XX,\ i \in \irg{1}{k}$ with weights $\ggl\in\Delta_k$, but $f(\x)< \sum_{i=1}^{k}\gl_i f(\x_i)$. Then, a function that agrees with $f$ on $\XX$ cannot be concave, so a concave tent cannot exist. 
		To argue sufficiency, we consider the graph of $f$ over $\XX$, i.e. 
		$
		\FF \coloneqq \lrbr{
			[
			f(\x), \ \x\T
			]\T\in\R^{n+1} \colon \x\in\XX 
		}.
		$
		By compactness of $\XX$ and continuity of $f$ relative to $\XX$, we have that $\FF$ is compact so that its convex hull is compact as well (see \cite[Corollary 5.33, p.185]{aliprantis_infinite_2006}). Define 
		\begin{align*}
			\hspace{-0.5cm}
			g(\x) \coloneqq \sup_{\mu}\lrbr{\mu \colon 
				\begin{bmatrix}
					\mu \\ \x
				\end{bmatrix}
				\in \conv(\FF)
			} = \sup_{\gl,\x_i}\lrbr{\sum_{i=1}^{k}\gl_if(\x_i)\colon \sum_{i=1}^{k}\x_i = \x,\ \ggl\in\Delta_k,\ \x_i\in\XX, \ i \in \irg{1}{k}},
		\end{align*}
		where we can set $k=n+2$ by Carathéodory's theorem. Then $g$ is a concave function by \cite[Theorem 5.3.]{rockafellar_convex_2015} and the supremum is attained by the compactness of $\XX$ and continuity of $f$. We will show that it agrees with $f$ on $\XX$. Assume to the contrary that there is an $\x\in\XX$ so that $g(\x)\neq f(\x)$.  
		We always have $g(\x)\geq f(\x)$, as $\ggl=\e_1$ and $\x_1=\x$ is feasible for the supremum so that this assumption implies $g(\x)>f(\x)$. But then $f(\x)< \sum_{i=1}^{k}\gl_i f(\x_i),$ while $\XX\ni\x =  \sum_{i=1}^{k}\gl_i  \x_i, \  \ggl\in\Delta_k, \ \x_i\in\XX, \ i \in\irg{1}{k}$ by attainment of the supremum,	in contradiction to \eqref{cond:Existence}. Thus, $g$ is a concave function that agrees with $f$ on $\XX$ and hence a concave tent. 
		%		Finally, to observe c), we argue that $g$ as constructed above is in fact the concave envelope of $f$ over $\XX$. Indeed, define the half line $\LL\coloneqq \lrbr{\gl\e_1\in\R^{n+1}\colon \gl\leq 0}$, and from an elementary argument we see that for the hypograph of $f$ over $\XX$ we have that its convex hull is 
		%		\begin{align*}
		%			\conv
		%			\left(
		%			\lrbr{
		%				\begin{bmatrix}
		%					\mu \\ \x
		%				\end{bmatrix} \colon \mu\leq f(\x),\ \x\in\XX 
		%			} 
		%			\right)
		%			= 
		%			\conv(\lrbr{\FF+\LL})
		%			=\conv(\FF)+\LL,
		%		\end{align*}  
		%		and the latter set is the hypograph of $g$. Concave functions are entirely described by their hypograph so that $g$ is identical to the concave envelope of $f$. It is closed since the compactness of $\FF$ implies that $\conv(\FF)$ is compact by \cite[Corollary 5.33, p.185]{aliprantis_infinite_2006} so that by \cite[Corollary 9.11]{rockafellar_convex_2015} the sum is also closed. A concave function is closed if its hypograph is closed by \cite[Theorem 7.1.]{rockafellar_convex_2015} so that the argument is complete. 
	\end{proof}
	The requirements on $f$ and $\XX$ for the necessary condition to be sufficient are quite drastic, and we hypothesize that much weaker assumptions could suffice. Indeed, \cite[Theorem 2.]{tawarmalani_convex_2002} discusses a much broader necessary and sufficient condition for the existence of a concave extensions of $f$ over a convex-set $\CC$ (not necessarily equal to $\conv(\XX)$). However, for the purposes of this text the above conditions suffice. Also, in the context of optimization such regularizations are quite common and we do not wish to delve too deep into convex analysis at this point. Now, we merely like to mention that the necessary condition is trivially fulfilled if there is no point in $\XX$ that is a convex combination of other points in $\XX$. This is the case if $\XX$ is contained in the set of extreme points in $\conv(\XX)$, which is the setup we will be concerned with exclusively. However, in \cref{sec:Super-derivatives based on conic duality} we will devise a strategy to work around the limitations of this special case so that our discussion is not limited to situations where condition \eqref{cond:Existence} is satisfied. Further, no such limitation is required for the following result, which summarizes the core motivation behind studying concave tents in the first place. 
	
	\begin{thm}\label{thm:ConcaveReformulations}
		Let $f$ and $\XX$ be as in \cref{def:ConcaveTent}. If a concave tent of $f$ over $\XX$, say $g$, exists, then for the optimization problems 
		$\inf_{\x\in\XX}\lrbr{ f(\x) }\mbox{ and } \inf_{\x\in\conv(\XX)} \lrbr{g(\x)}$ the following hold:
		\begin{enumerate}
			\item[a)] Every strict local minimizer of the latter problem is a strict local minimizer of the former.
			\item[b)] We have $\inf_{\x\in\XX}\lrbr{f(\x)} = \inf_{\x\in\conv(\XX)} \lrbr{g(\x)}$, and they have the same global minimizers.
		\end{enumerate}
	\end{thm}
	\begin{proof}
		We first prove a), so let $\x\in\conv(\XX)$ be a local minimizer of the concave optimization problem, i.e. $g(\x)<g(\y),\ \forall \y\in\left(\conv(\XX)\cap \BB_{\eps}(\x)\right)\setminus\lrbr{\x}$ for some $\eps>0$. If $\x$ is an extreme point of $\conv(\XX)$, then $\x\in\XX$ and we get $f(\x)<f(\y),\ \forall \y\in\left(\XX\cap \BB_{\eps}(\x)\right)\setminus\lrbr{\x}$ as desired, since $f(\y) = \g(\y), \ \forall \y\in\XX$ and $\XX\cap \BB_{\eps}(\x)\subseteq\conv(\XX)\cap \BB_{\eps}(\x)$. So assume $\x$ is not an extreme point. Then there are $\x_i\in\conv(\XX)\cap\BB_{\eps}(\x)\setminus\lrbr{\x},\ i \in\irg{1}{k},\ \ggl\in\Delta_k$ such that $\x =  \sum_{i=1}^{k}\gl_i\x_i$ (such points exist in $\conv(\XX)\setminus\lrbr{\x}$ and can be contracted towards $\x$ until they become members of $\BB_{\eps}(\x)$).
		We have $\g(\x)\geq \sum_{i=1}^{k}\gl_ig(\x_i)\geq g(\x_{i^*})$ for some $i^*\in\irg{1}{k}$ since an average cannot be smaller than its smallest constituent, but this is in contradiction to strict local optimality. To prove b), if $\x\in\XX$ then $f(\x) = g(\x)$ and $\x\in\conv(\XX)$ so that "$\geq$" is immediate. Conversely, if $\x\in\conv(\XX)\setminus\XX$ then $\x= \sum_{i=1}^{k}\gl_i\x_i, \ \ggl\in\Delta_k, \ \x_i\in\XX, \ i \in\irg{1}{k}$ and $g(\x) = g(\sum_{i=1}^{k}\gl_i\x_i)\geq\sum_{i=1}^{k}\gl_ig(\x_i)$ and there must be at least one $\x_i\in\XX$ so that $g(\x)\geq g(\x_i) = f(\x_i)$. Thus, the infima are identical and the final statement follows since global minimizers of the concave problem, if they exist, are attained in $\XX$ by the above argument.
	\end{proof}
	
	\begin{rem}
		In case $\XX\subseteq \lrbr{0,1}^n$, the definition of strict local minimizer can be applied, but is not meaningful as every point in $\XX$ would be a strict local minimizer. However, the above theorem is not restricted to such sets and point b) is still meaningful even for discrete sets, so we chose not to exclude, nor treat separately the discrete case.
	\end{rem}

	The core message of this theorem is the following: concave tents allow us to employ methods from concave minimization to tackle $\inf_{\x\in\XX} f(\x)$, where we may exploit knowledge on $\conv(\XX)$. Optimizing the concave tent globally solves the original problem, and strict local solutions yield upper bounds that cannot be improved locally. Thus, well-studied concave optimization tools as well as results on $\conv(\XX)$ can potentially be applied in a new area. 
	
	\section{A versatile class of functions}\label{sec:A versatile class of functions}
	The mere existence of a concave tent does not imply that it can be easily constructed. Indeed, for this purpose, we have to make further assumptions on the feasible set $\XX$ and the function $f$, the latter of which will be discussed in this section. Namely, henceforth we focus on functions that fulfill the following:
	
	\begin{asm}\label{asm:CompactU} 
		The objective function $f$ can be written as  
		\begin{align*}
			f(\x) 
			= h(\x)+
				\x\T\Ab\x+2\a\T\x+\max_{\u \in \R^q}  
				\lrbr{2\u\T\Bb\x+\u\T\Cb\u+2\cc\T\u 
				\colon 
				\u\in\UU
			},
		\end{align*}			
		for a closed proper concave function $h$, a compact convex nonempty set  $\UU\subseteq\R^q$, matrices $\Ab\in\SS^n,\ \Bb\in\R^{q\times n},\ \Cb\in\SS^q$ and vectors $\a\in\R^n,\ \cc\in\R^q$. 			
	\end{asm} 
	
	Let us briefly discuss the versatility of this class of functions. Firstly, note that the supremum term is always a closed convex function of $\x$ since it is the pointwise supremum of functions that are linear in $\x$. Indeed, by \cite[Theorem 12.1.]{rockafellar_convex_2015} any closed, convex function can be written as a pointwise supremum of linear functions so that \cref{asm:CompactU} excludes convex functions that are not closed, which play a minor role in convex optimization anyway. This also might give the appearance that including the term $\u\T\Cb\u$ is spurious, however, some convex functions are naturally given is in this form, e.g., those that appear in robust optimization as discussed shortly. 
	
	For an arbitrary closed convex function $c$, we can often recover a representation that conforms with \cref{asm:CompactU}. From the convex conjugate of a convex function $c$ given by $c^*(\u) = \sup_{\x\in\R^n}\lrbr{\u\T\x - c(\x)}$ we get its biconjugate characterization 
	\begin{align*}
		\hspace{-1cm}
		c(\x)  		
		= 
		\sup_{\u\in\R^n}\lrbr{\x\T\u - c^*(\u)} 
		= 
		\sup_{\substack{\u\in\R^n,\\ u_0\in\R}}
		\lrbr{
			\x\T\u - u_0
			\colon 
			\left[\u\T,u_0\right]
			\in\UU
		}, \ 
		\UU
		\coloneqq
		\lrbr{
			\begin{bmatrix}
				\u \\ u_0
			\end{bmatrix}\in\R^{n+1}
			\colon 
			f^*(\u)\leq u_0
		},
	\end{align*}
	so that 
	$f(\x) \coloneqq h(\x) + \x\T\Ab\x+2\a\T\x + c(\x)$ can be rewritten in the required form. However, the set $\UU$ as defined above is not compact since $u_0$ is not constrained from above and even if it were, the effective domain of $c^*$ may be unbounded. We address this issue in greater detail in \cref{apx:RegularizingConveFunctions}, where we will show how to replace $c$ by surrogate function, the Pasch-Hausdorff envelope, that agrees with $f$ on $\XX$ and conforms with \cref{asm:CompactU}. Consequently, the model in \cref{asm:CompactU} encompasses many instances of difference-of-convex optimization.
	
	The second source of minimization problems where the functions conform with \cref{asm:CompactU} stems from robust optimization especially robust optimization with quadratic indices. In such models, $\u$ would take the role of the uncertainty parameter and $\UU$ the role of the uncertainty set. The special case $f(\x)\coloneqq \sup_{\u \in \UU} \lrbr{(\a+\Bb\T\u)\x+ \cc\T\u}$ is the standard worst-case evaluation in robust linear optimization (see e.g. \cite{ben-tal_robust_2009}).  Examples of robust optimization models where the uncertainty parameter appears quadratically are given by adjustable robust optimization under quadratic decision rules or under affine decision rules but with uncertain recourse and by robust quadratic optimization (see e.g.\ \cite{ben-tal_adjustable_2004} and \cite{bomze_optimization_2023} and references therein).
	
	To the best of our knowledge an exact characterization of functions that fulfill \cref{asm:CompactU} is not known and we defer an exhaustive investigation to future research and close this section with the following useful result:
	\begin{prop}\label{prop:continuity}
		Under \cref{asm:CompactU}, the supremum in the definition of $f$ is locally Lipschitz in $\x$, and $f$ itself is upper semicontinuous and nowhere equal to $\infty$. 
	\end{prop}
	\begin{proof}
		The first statement follows from \cite[ Theorem (2.1)]{clarke_generalized_1975}. Thus, $f$ is the sum of a continuous function, a quadratic function and the closed proper concave function $h$, hence itself upper semicontinuous and smaller than $\infty$ everywhere.  
	\end{proof}
	
	\section{Constructing concave tents via conic optimization}
	In this section, we will derive concave tents as optimal value functions of certain conic optimization problems. We divide our discussion into three parts: 
	to ease the reader into the subject, we will present a construction of concave tents for functions that are simpler than the ones presented in \cref{asm:CompactU} and $\XX\subseteq \lrbr{0,1}^n$. After that, we will derive concave tents where $f$ conforms with \cref{asm:CompactU} and $\XX$ is, more generally, a set that contains only extreme points of its convex hull. Finally, we will present a double lifting strategy that will allow us to bypass the restrictions of the necessary condition in \cref{thm:ExistenceofCT} so that we can construct concave reformulations over general sets $\XX$.

	\subsection{Concave tents of convex functions over subsets of $\lrbr{0,1}^n$}		
	In this subsection, we will discuss concave tents over $\XX\subseteq \lrbr{0,1}^n$ and functions that are of the form $f(\x) = \sup_{\u \in \UU} \u\T\Bb\x+\cc\T\u$ for some compact convex $\UU\subseteq \R^q$, i.e.\ we restrict the discussion to closed convex functions. Note, that due to compactness we can assume w.l.o.g. that $\|\u\|^2\leq 1,\ \forall \u\in\UU$ so that $\Ib\bullet\Ub\leq 1$ whenever $(\u,\Ub)\in\GG(\UU)$. We will make use of the following two lemmas: 
	
	\begin{lem}\label{lem:Concavity}
		Let $g\colon \R^n\rightarrow \R$ be given by 
		$
		g(\x)
		\coloneqq
		\sup_{\y\in\R^k}
		\lrbr{
			\p\T\x+\q\T\y 
			\colon 
			\begin{bmatrix}
				\x\T,&\y\T
			\end{bmatrix}\T
			\in\PP
		},
		$
		for some vectors $\p\in\R^n$ and $\q\in\R^k$ and a convex set $\PP\subseteq \R^{n+k}$, then $g$ is a concave function. 
	\end{lem} 
	\begin{proof}
		Assume that $\lrbr{\x_1,\x_2}\subseteq \R^n$ such that 
		$
		\exists \y_i\colon 
		\begin{bmatrix}
			\x_i\T,&\y_i\T
		\end{bmatrix}\T
		\in\PP,\ i = 1,2. 
		$
		Then, by convexity $\y(\gl) \coloneqq \gl\y_1+(1-\gl)\y_2$ for $\gl\in[0,1],$ is a feasible solution for $g(\x(\gl))$ where $\x(\gl)\coloneqq \gl\x_1+(1-\gl)\x_2$ and, thus, yields a lower bound so that 
		$
		g(\x(\gl))= g(\gl\x_1+(1-\gl)\x_2) \geq \p\T\x(\gl)+\q\T\y(\gl) = \gl g(\x_1)+(1-\gl)g(\x_2).
		$
		If at least one of $\x_1,\x_2$ is such that no suitable $\y_i,\ i = 1,2$ exists then the inequality holds trivially, since the respective value for $g$ is $-\infty$. As an alternative geometrical argument one can construct the hypograph of $g$ defined by the inequality $t\leq g(\x)$ (where $t$ is the vertical coordinate) as the hypograph $t\leq \p\T\x+\q\T\y$ of the objective function over $\PP$ projected onto the $(\x,t)$-coordinates. The details are left to the reader.      
	\end{proof}

	\begin{lem}\label{lem:ExtremeRank1}
		The following holds
		\begin{enumerate}
			\item[a)] For $\x\in\R^n,\ \lrbr{\Ab,\Bb}\subset \SS_+^n$ we have that $\x\x\T  = \Ab + \Bb$ implies $\Ab = \ga \x\x\T,\ \Bb = \gb \x\x\T$ where $\ga,\ \gb$ are nonnegative and $\ga+\gb = 1$ .
			\item[b)] For $\lrbr{\x,\y} \subset \R^n$ we have that $\x\x\T= \y\y\T$ implies $\y\in\lrbr{\x,-\x}$.  
		\end{enumerate}
	\end{lem}
	\begin{proof}
		These are well-known facts, that are proved in, for example, \cite[Proposition 11.]{bomze_optimization_2023}.
	\end{proof}
	
	We are now ready to state the main theorem of this section.  
	\begin{thm}\label{thm:ConcaveTentOver01}
		Let $f(\x) = \sup_{\u\in \UU} \u\T\Bb\x+ \cc\T\u$, such that $\UU\subseteq \BB_{1}(\oo)\subseteq \R^q$ is convex and compact and let $\XX\subseteq \lrbr{0,1}^n$. Further, let $\CC\subseteq\SS^{n+q+1}$ be a closed convex set such that
		$\SS_+^{n+q+1} \supseteq \CC \supseteq \GG(\UU\times\XX)$.
		Define
		\begin{align*}
			g(\x) \coloneqq \sup_{\Psi,\u,\Ub,\Xb} \lrbr{\Bb\T\bullet \Psi + \cc\T\u \colon
				\begin{pmatrix}
					1  & \u\T & \x\T \\
					\u & \Ub  & \Psi\T \\
					\x & \Psi & \Xb
				\end{pmatrix} \in \CC, \
				\u\in\UU,\ 
				\begin{array}{r}
					\diag(\Xb) = \x, \\
					\Ib\bullet\Ub \leq 1,
				\end{array}
			}.
		\end{align*} 
		Then $g$ is a concave tent of $f$ over $\XX$.
	\end{thm}	
	\begin{proof}
		The concavity of $g$ is a direct consequence of \cref{lem:Concavity}. To prove that $g$ and $f$ agree on $\XX$, we will use the fact that both functions are given in terms of optimization problems, which we will show to be equivalent whenever $\x\in\XX$.
		To prove "$\leq$", let $\u\in \UU$ and $\y \coloneqq \left(1,\u\T,\x\T\right)$. Then $\y\y\T$ is a feasible matrix for the problem defining $g$, where $\diag(\Xb)= \diag(\x\x\T) = \x $ since $\x\in\XX$, $\Ib\bullet\Ub = \u\T\u \leq 1$	since $\u\in\UU\subseteq\BB_{1}(\oo)$ and $\y\y\T\in\GG(\UU\times\XX)\subseteq\CC$ be the same inclusions. Finally, 	$\Psi = \x\u\T$  so that 
		for the objective function we get $\Bb\T\bullet\Psi+\cc\T\u= \Bb\T\bullet\x\u\T+\cc\T\u = \u\T\Bb\x+\cc\T\u$ as desired. 
		
		For the other direction let $\left(\Psi,\u,\Ub,\Xb\right)$ be feasible for the conic problem for our chosen $\x\in\XX$. By feasibility, we know that 
		\begin{align*}
			\begin{bmatrix}
				1 & \x\T \\ \x & \Xb
			\end{bmatrix}\in\SS^{n+1}_+ 
			\implies  \Xb-\x\x\T \in\SS_+^n  \implies 
			\Xb = \x\x\T+\Mb, \ \Mb\in\SS_+^n.
		\end{align*}
		Further, since $\x\in\XX\subseteq \lrbr{0,1}^n$ we have $ \diag(\x\x\T) -\x = \oo$. But then 
		\begin{align*}
			\oo = \diag(\Xb)-\x = \diag(\x\x\T+\Mb)-\x = \diag(\x\x\T)-\x+\diag(\Mb) = \oo+\diag(\Mb),
		\end{align*}
		so that $\diag(\Mb) = \oo$, which together with $\Mb\in\SS_+^n$ implies $\Mb = \Ob$ and we get $\Xb = \x\x\T$.
		Further, since $\SS_+^n = \conv\left(\lrbr{\w\w\T\colon \w\in \R^n}\right)$ we know that 
		\begin{align}\label{eqn:SDPDecomposition}
			\begin{bmatrix}
				1  & \u\T & \x\T \\
				\u & \Ub  & \Psi\T \\
				\x & \Psi & \Xb
			\end{bmatrix} = 
			\sum_{i=1}^{k}  
			\begin{bmatrix}
				\gl_i^2  & \gl_i\u_i\T & \gl_i\v_i\T \\
				\gl_i\u_i & \u_i\u_i\T  & \u_i\v_i\T \\
				\gl_i\v_i & \v_i\u_i\T & \v_i\v_i\T
			\end{bmatrix} \mbox{ for some }
			\begin{bmatrix}
				\gl_i \\ \u_i \\ \v_i
			\end{bmatrix} \in\R^{n+q+1},\ i \in\irg{1}{k}.
		\end{align}
		After plugging in $\Xb = \x\x\T$, inspecting the respective submatrix yields 
		\begin{align*}
			\begin{bmatrix}
				1\\ \x
			\end{bmatrix}
			\begin{bmatrix}
				1\\ \x
			\end{bmatrix}\T  = \sum_{i=1}^{k} 
			\begin{bmatrix}
				\gl_i\\ \v_i
			\end{bmatrix}
			\begin{bmatrix}
				\gl_i\\ \v_i
			\end{bmatrix}\T \ \implies 
			\begin{bmatrix}
				\gl_i\\ \v_i
			\end{bmatrix}
			\begin{bmatrix}
				\gl_i\\ \v_i
			\end{bmatrix}\T = \mu_i 
			\begin{bmatrix}
				1\\ \x
			\end{bmatrix}
			\begin{bmatrix}
				1\\ \x
			\end{bmatrix}\T,\ \mmu\in\Delta_k, 
		\end{align*}
		by \cref{lem:ExtremeRank1} a). Then, we can see from \cref{lem:ExtremeRank1} b) that
		\begin{align*}
			\begin{bmatrix}
				\gl_i\\ \v_i
			\end{bmatrix}	 = \pm
			\begin{bmatrix}
				\sqrt{\mu_i}\\ \sqrt{\mu_i}\x
			\end{bmatrix}, \mbox{ so that } \ \u  = \sum_{i=1}^{k}\pm\sqrt{\mu_i}\u_i, \mbox{ by (\ref{eqn:SDPDecomposition}),}
		\end{align*}
		and we finally get $\Psi = \sum_{i=1}^{k}\v_i\u_i\T = \sum_{i=1}^{k}\pm\sqrt{\mu_i}\x\u_i\T = \x\sum_{i=1}^{k}\pm\sqrt{\mu_i}\u_i\T = \x\u\T,$
		and $\u\in\UU$ by feasibility so that it is feasible for the problem defining $f$, with 
		$ \Bb\T\bullet\Psi+\cc\T\u =\Bb\T\bullet\x\u\T+\cc\T\u = \u\T\Bb\x+\cc\T\u$, which completes the proof.			  				
	\end{proof}
	
	We will now give an illustration of the above theorem in a simple example. 
	
	\begin{exmp}\label{exmp:ConcaveTents}
		Consider the following nonconvex optimization problem 
		\begin{align}\label{eqn:ExmpProblem}
			\min_{x\in\XX} \max_{(u_1,u_2)\T\in\UU} u_1+u_2x, \mbox{ with }\XX\coloneqq\lrbr{0,1}, \ \UU\coloneqq 
			\lrbr{
				\begin{bmatrix}
					1\\ -8
				\end{bmatrix},
				\begin{bmatrix}
					-2\\ 2
				\end{bmatrix}
			}.
		\end{align}
		The inner maximization problem constitutes a convex function in $x$, say $f$, that is the pointwise maximum of two affine functions. It evaluates to $f(0) = 1,\ f(1) = 0$ and $f(3/10) = -7/5$, which is its global minimum over $\conv(\XX)=[0,1]$. By linearity with respect to $(u_1,u_2)$ we can replace $\UU$ with its convex hull 
		\begin{align*}
			\conv(\UU) = 
			\lrbr{
				\begin{bmatrix}
					u_1\\ u_2
				\end{bmatrix} = u	
				\begin{bmatrix}
					1\\ -8
				\end{bmatrix}+ (1-u)
				\begin{bmatrix}
					-2\\ 2
				\end{bmatrix}\colon u\in[0,1]
			},
		\end{align*}
		so that we can substitute $(u_1,u_2)$ via an affine expression in $u$ after which (\ref{eqn:ExmpProblem}) can be rewritten as 
		\begin{align*}
			\min_{x\in\lrbr{0,1}} \max_{u\in[0,1]} 3u+2x-10ux-2,
		\end{align*}
		To construct a concave tent of $f$ over $\lrbr{0,1}$ we need an appropriate subset $\CC$ of $\SS^3_+$ that is also a superset of $\GG([0,1]\times\lrbr{0,1})$. We propose 
		\begin{align*}
			\CC \coloneqq 
			\lrbr{
				\begin{bmatrix}
					1 & u & x \\
					u & U & \psi\\
					x &\psi& X
				\end{bmatrix} \in\SS^3_+, \colon 
				(x,X)\in\GG(\lrbr{0,1}),\
				(u,U)\in\GG([0,1]),\
				\psi\geq 0 
%				\begin{array}{rl}
%					(x,X)\in&\hspace{-0.2cm}\GG(\lrbr{0,1}),\\
%					(u,U)\in&\hspace{-0.2cm}\GG([0,1]),\\
%					\psi\geq&\hspace{-0.2cm} 0
%				\end{array}				
			} ,
		\end{align*}
		since we have the simple characterizations (which are discussed for example in \cite{burer_gentle_2015}):
		\begin{align*}
			\GG(\lrbr{0,1}) = 
			\lrbr{
				\begin{bmatrix}
					1 & x \\
					x & X
				\end{bmatrix}\in\SS^2_+ \colon 
				\begin{array}{l}
					X\leq x\leq 1\\
					X = x 
				\end{array}
			}, \ 
			\GG([0,1]) = 
			\lrbr{
				\begin{bmatrix}
					1 & u \\
					u & U
				\end{bmatrix}\in\SS^2_+ \colon U\leq u\leq 1
			},
		\end{align*}
		and $\psi\geq 0$ whenever $\psi = xu$ with $x\in\lrbr{0,1}$ and $u\in[0,1]$ so that, clearly, we have $\SS^3_+\supseteq \CC \supseteq \GG([0,1]\times\lrbr{0,1})$ as required. In addition we have that $(u,U)\in\GG([0,1])$ implies $1\bullet U\leq 1$ as required but the constraint is redundant on $\CC$. We are now ready to construct our concave tent as 
		\begin{align}\label{eqn:ConcaveTentExmp}
			g(x) \coloneqq  \sup_{u,U,\psi}\lrbr{ 3u + 2x - 10\psi - 2 \colon 
				\begin{bmatrix}
					1 & u & x \\
					u & U & \psi\\
					x &\psi& x
				\end{bmatrix}\in\SS^3_+, \ U\leq u\leq 1, \ 0 \leq \psi 
			}.
		\end{align}
		Note, that due to the equality constraint in the description of  $\GG(\lrbr{0,1})$ we were able to eliminate $X$, rendering the respective constraints vacuous. In addition, the constraint $x\leq 1$ is then implied by the psd (positive semidefinite) constraint, so that it can be omitted as well. 
		
		We will show that $g$ is indeed a concave tent as proclaimed in \cref{thm:ConcaveTentOver01}. First, note that for any feasible solution, we can always increase $U$ until it equals $u$ since the matrix will not change definiteness. But then, the constraint $u\leq 1$ is redundant since the psd constraint already implies $u\geq u^2$, as all principal minors must be nonnegative. We claim that the optimal solution is given by $u = 1-x,\  \psi = 0$ with optimal value $1-x$. We only need to check the psd constraint to argue feasibility, which is easily done by observing that all principal minors are indeed nonnegative. To prove optimality we inspect the dual problem given by 
		\begin{align} \label{eqn:SmallDual}
			\inf_{\ga,\gb,\gc,\gd,\ge,\gp}\lrbr{x(2\gc+\gp+2) + \ga -2
				\colon
				\begin{bmatrix}
					\ga & \gb & \gc \\
					\gb & \gd & \ge \\
					\gc & \ge & \gp
				\end{bmatrix}\in\SS^3_+, \ 3+2\gb+\gd = 0, \ 2 \ge \leq 10
			}.
		\end{align}
		By weak conic duality, it is sufficient to find a feasible dual solution that attains the proclaimed optimal value. Indeed, for $\ga=\gd=\ge=\gp = 3, \ \gb=\gc= -3$, we find the constraints fulfilled, where positive definiteness can again be certified by the principal minors, and the optimal value is indeed $1-x$. This is not only a concave tent but it tracks the concave envelope of $f$ over the unit interval so we could not improve the upper bounds if we strengthened $\CC$. It is, however, interesting to see what would happen if the constraints were relaxed instead.
		
		Indeed, if we merely omit the constraints $1\geq u \geq U$, a weaker concave tent would be given by   
		\begin{align}\label{eqn:ConcaveTentExmp2}
			\tilde{g}(x) \coloneqq \sup_{u,U,\psi}\lrbr{ 3u + 2x - 10\psi - 2 \colon 
				\begin{bmatrix}
					1 & u & x \\
					u & U & \psi\\
					x &\psi& x
				\end{bmatrix}\in\SS^3_+ \colon U\leq 1,\ 0 \leq \psi 
			}.
		\end{align}
		To solve the sdp, again consider, that it is always possible to set $U=1$ without changing the optimal value. Also, since $\psi=0$ was optimal for $g$, and the eliminated constraint did not involve $\psi$ we can infer that this choice for $\psi$ is still optimal. Then, we relax the psd condition to merely requiring that the determinant be nonnegative, i.e. $x - u^2 x - x^2 \geq 0$.
		Since $x\geq 0$, the latter constraint reduces to $1-x \geq u^2$,
		so that the optimal choice for the relaxation is given by $u = \sqrt{1-x}$. However, it is easily checked, again by enumeration of all principal minors, that this solution is feasible for (\ref{eqn:ConcaveTentExmp2}) and, hence, optimal. In total, we have shown that $\tilde{g}(x) = 3\sqrt{1-x}+2x-2$, which is indeed concave but strictly greater than $g$ over the unit interval, except for the endpoints, where both functions are equal to $f$.		
	\end{exmp}

	\subsection{Concave tents over sets of extreme points}
	
	In this section, we will generalize and strengthen our results from the previous section. The sufficient conditions in the theorem below and their relation to the conditions in \cref{thm:ExistenceofCT} are discussed in the sequel. 	
	\begin{thm}\label{thm:ConcaveTentOverExtreme}
		Let $f$ be such that \cref{asm:CompactU} holds and let  $\XX\subseteq\R^n$. 
		Further, let $\CC\subseteq\SS^{n+q+1}$ be a closed convex set such that 
		$\SS^{n+q+1}_+ \supseteq \CC \supseteq \GG(\UU\times\XX).$
		Finally let $\AA\colon \SS^n\times\R^n \rightarrow \R^m$ be a linear function and $\b\in\R^m$, such that 
		\begin{align}
			\lrbr{
				\begin{bmatrix}
					1 & \x\T \\ \x & \Xb
				\end{bmatrix}\in\SS^{n+1}_+ \colon \AA(\Xb,\x) = \b
			}&\supseteq \GG(\XX), \label{eqn:Nec1}\\
			\lrbr{
				\Xb\in\SS^{n}_+ \colon \AA(\Xb,\oo) = \oo
			}& =  \lrbr{\Ob}\label{eqn:Nec2}.
		\end{align}
		Define
		\begin{align}\label{eqn:CT1}
			g(\x) 
			\coloneqq 
			h(\x)+
			\sup_{\Psi,\u,\Ub,\Xb} \lrbr{
				\begin{array}{l}
					\Ab\bullet\Xb+2\a\T\x+2\Bb\T\bullet\Psi+\Cb\bullet\Ub+2\cc\T\u 
					\colon
					\\
					\begin{bmatrix}
						1  & \u\T & \x\T \\
						\u & \Ub  & \Psi\T \\
						\x & \Psi & \Xb
					\end{bmatrix} \in \CC, \
					\begin{bmatrix}
						1  & \u\T \\
						\u & \Ub   
					\end{bmatrix}\in\GG(\UU),\ 
					\AA(\Xb,\x) = \b
				\end{array}				
			}.
		\end{align} 
		Then $g$ is a concave tent of $f$ over $\XX$, closed and proper if $\XX\neq\emptyset$. 
		If $\Cb = \Ob$ the statement remains true even if $\GG(\UU)$ is replaced by any compact convex set $\GG_{\UU}\supseteq \GG(\UU)$ for which $(\u,\Ub)\in\GG_{\UU} \implies \u\in\UU$.
	\end{thm}
	\begin{proof}
		Again, \cref{lem:Concavity} and the fact that the sum of concave functions is concave show that $g$ is a concave function. In addition, if a function $\hat{g}$ is a closed proper concave tent of $f$ over $\XX$ then $h+\hat{g}$ is a closed proper concave tent of $f+h$ over $\XX$ since all of these attributes are preserved under addition, thus it suffices to consider the case where $h$ is the zero function.  
		
		We will argue that $g$ is a closed and proper concave function if $\XX\neq \emptyset$. Since $\XX\neq\emptyset\neq\UU$ by assumption there is an $\y\coloneqq [1,\x\T,\u\T]\T\colon \x\in\XX,\ \u\in\UU$ such that the matrix $\y\y\T$ gives a feasible matrix block for the supremum so that it is larger than $-\infty$ in at least one point, the first requirement for properness. The boundedness of the feasible set will supply the boundedness of the supremum from above and therefore the second requirement for properness. 
		Note that the feasible set is closed, since  $\CC$ and $\GG(\UU)$ (or $\GG_{\UU}$) are closed and all constraints are linear so that the feasible set is an intersection of closed sets, hence closed. 
		
		Closed convex sets are bounded if and only if they do not contain a ray (see \cite[Theorems 8.3 and 8.4]{rockafellar_convex_2015}). So assume that $(\Psi,\u,\Ub,\Xb)\coloneqq (\Psi_0,\u_0,\Ub_0,\Xb_0)+\gl(\bar\Psi,\bar\u,\bar\Ub,\bar\Xb)
		$ is feasible for $\g(\x)$ for some $\x$ and for all $\gl\geq 0$. Then $\bar\u=\oo$ and $\bar\Ub=\Ob$, since otherwise $\GG(\UU)$ (or $\GG_{\UU})$ would be unbounded contrary to our assumptions.  
		It follows from $\AA(\Xb,\x) = \AA(\Xb_0,\x)+\gl\AA(\bar{\Xb},\oo) = \b$ and $\AA(\Xb_0,\x) = \b$ (from the case $\gl = 0$) that $\AA(\bar{\Xb},\oo) = \oo$ so that $\bar{\Xb} = \Ob$ by (\ref{eqn:Nec2}). Finally $(\Psi)_{ij}^2\leq(\Xb)_{ii}(\Ub)_{jj},\ \forall (i,j)\in\irg{1}{n}\times \irg{1}{q}$ by positive semidefiniteness which implies $\bar\Psi= \Ob$ establishing boundedness and, hence, properness. 
		
		Closedness of $g$ now follows because the dual of the supremum in $g$ is exact whenever the primal has a compact feasible set \cite[Proposition 2.8]{shapiro_duality_2001}. This dual is a pointwise infimum of linear functions in $\x$, since $\x$ only appears as a right-hand side in the constraints and in a constant in the objective (see \cref{sec:Super-derivatives based on conic duality} for an explicit presentation of the dual). Hence, the supremum, and therefore $g$, is a closed concave function by \cite[Theorem 12.1.]{rockafellar_convex_2015}.
		
		To prove that $g$ and $f$ agree on $\XX$ we employ the same strategy as in \cref{thm:ConcaveTentOver01}, where the "$f\geq g$"-direction will be more challenging this time around. We omit the other direction since it is an almost verbatim repetition of the argument laid out in the proof of \cref{thm:ConcaveTentOver01}. So let $\left(\Psi,\u,\Ub,\Xb\right)$ be feasible for the conic problem in $g(\x)$ for some $\x\in\XX$. By feasibility, we know that 
		\begin{align*}
			\begin{bmatrix}
				1 & \x\T \\ \x & \Xb
			\end{bmatrix}\in\SS^{n+1}_+ 
			\implies  \Xb-\x\x\T \in\SS_+^n  \implies 
			\Xb = \x\x\T+\Mb, \mbox{ for some } \Mb\in\SS_+^n.
		\end{align*}
		Further, since $\x\in\XX$ we have $\AA(\x\x\T,\x) = \b$ by (\ref{eqn:Nec1}). But then 
		\begin{align*}
			\b = \AA(\Xb,\x) = \AA(\x\x\T+\Mb,\x+\oo) = \AA(\x\x\T,\x)+\AA(\Mb,\oo) = \b+\AA(\Mb,\oo),
		\end{align*}
		so that $\Ab(\Mb,\oo) = \oo$, which by (\ref{eqn:Nec2}) implies $\Mb = \Ob$ so that $\Xb = \x\x\T$. Repeating the argument in the proof of \cref{thm:ConcaveTentOver01} yields $\Psi = \x\u\T$. In case $\Cb= \Ob$ we are done at this point since $(\u,\Ub)\in\GG_{\UU} \implies \u\in\UU$ so that $\u$ is a feasible solution of the supremum defining $f(\x)$ with the same objective function value as $g(\x)$. 
		Otherwise, consider that 
%		\begin{align*}
%			g(\x) = 
%			\sup_{\u\in\R^m,\, \Ub\in\SS^m} \lrbr{
%				\begin{array}{l}
%					\Ab\bullet\x\x\T + 2\a\T\x + 2\Bb\T\bullet\x\u\T + \Cb\bullet\Ub + 2\cc\T\u 
%					\colon
%					\\
%					\begin{bmatrix}
%						1  & \u\T & \x\T \\
%						\u & \Ub  & \u\x\T \\
%						\x & \x\u\T & \x\x\T
%					\end{bmatrix} \in \CC, \
%					\begin{bmatrix}
%						1  & \u\T \\
%						\u & \Ub   
%					\end{bmatrix}\in\GG(\UU),\ 
%					\AA(\x\x\T,\x) = \b
%				\end{array}				
%			},
%		\end{align*} 
%		and we further argue that 
		\begin{align}\label{eqn:Intermediate}
			g(\x) 
			&= 
			\Ab\bullet\x\x\T + 2\a\T\x +
			\sup_{\u\in\R^q,\, \Ub\in\SS^q} \lrbr{
				\begin{array}{l}
					\Cb\bullet\Ub + 2(\cc+\Bb\x)\T\u 
					\colon
					\\
					\begin{bmatrix}
						1  & \u\T & \x\T \\
						\u & \Ub  & \u\x\T \\
						\x & \x\u\T & \x\x\T
					\end{bmatrix} \in \CC, \
					\begin{bmatrix}
						1  & \u\T \\
						\u & \Ub   
					\end{bmatrix}\in\GG(\UU)  
				\end{array}				
			}\nonumber
			\\ 
			&\leq
			\Ab\bullet\x\x\T + 2\a\T\x +
			\sup_{\u\in\R^q,\, \Ub\in\SS^q} \lrbr{
				\begin{array}{l}
					\Cb\bullet\Ub + 2(\cc+\Bb\x)\T\u 
					\colon						
					\begin{bmatrix}
						1  & \u\T \\
						\u & \Ub   
					\end{bmatrix}\in\GG(\UU)  
				\end{array}				
			}.
		\end{align} 
		Indeed, the equation holds  since we can pull constants out of the optimization problem and erase the now vacuous constraint $\AA(\x\x\T,\x)=\b$. The inequality holds, because we removed a constraint.
%		For the converse, we need to show that the constraint $(\u,\Ub)\in\GG(\UU) $ implies that the extended matrix block is indeed in $\CC$. Since $\UU$ is compact we get that $(\u,\Ub) = \sum_{i=1}^{k}\gl_i(\u_i,\u_i\u_i\T),\ \u_i\in\UU,\ i \in\irg{1}{k},\ \ggl\in\Delta_k$ for some $k\in\N$ so that 
%		\begin{align*}
%			\begin{bmatrix}
%				1  & \u\T & \x\T \\
%				\u & \Ub  & \u\x\T \\
%				\x & \x\u\T & \x\x\T
%			\end{bmatrix}  = 
%			\sum_{i=1}^{k}
%			\gl_i
%			\begin{bmatrix}
%				1  & \u_i\T & \x\T \\
%				\u_i & \u_i\u_i\T   & \u_i\x\T \\
%				\x & \x\u_i\T & \x\x\T
%			\end{bmatrix} \in \GG(\UU\times\XX)\subseteq \CC,
%		\end{align*}
%		so that the second equation is established. 
		But by \eqref{thm:chlifting} the expression in \eqref{eqn:Intermediate} is identical to $f$, which proves $f\geq g$.  
	\end{proof}	
	
	\subsubsection{Discussion of the sufficient conditions}

	In this section we will show that conditions (\ref{eqn:Nec1}) and (\ref{eqn:Nec2}) are not vacuous and investigate how they relate to the conditions in \cref{thm:ExistenceofCT}. Note, that any linear function $\AA \colon \SS^n\times \R^n\rightarrow \R^m$ has the following representation 
	\begin{align}\label{eqn:AARepresentation}
		\AA(\Xb,\x) = 
		\begin{bmatrix}
			\Ab_1\bullet\Xb + \a_1\T\x\\
			\vdots \\
			\Ab_m\bullet\Xb + \a_m\T\x
		\end{bmatrix}, \mbox{ for some } \Ab_i\in\SS^n, \ \a_i\in\R^n , \ i \in\irg{1}{m}.
	\end{align}
	We can prove the following equivalence by standard arguments from convex analysis.
	\begin{lem}\label{lem:PSDsum}
		Let $\AA\colon\SS^n\times\R^n \rightarrow \R^m$ be of the form (\ref{eqn:AARepresentation}). Condition (\ref{eqn:Nec2}) holds if and only if $\sum_{i=1}^{m}~\gl_i~\Ab_i\in~\SS^n_{++}$ for some $\ggl\in\R^m$.  
	\end{lem}
	\begin{proof}
		If $\ggl\in\R^m$ exists as required, then we have that  $\Ab(\Xb,\oo) = \oo$ implies $$0 = \ggl\T\oo = \ggl\T\AA(\Xb,\oo) = \sum_{i=1}^m\gl_i\Ab_i\bullet \Xb =\left(\sum_{i=1}^m\gl_i\Ab_i\right)\bullet \Xb,$$ which implies that $\Xb =\Ob$ whenever it is positive semidefinite since $\SS_+^n$ is pointed. For the converse, assume to the contrary that $\lrbr{\sum_{i=1}^{m}\gl_i\Ab_i\colon \ggl\in\R^m}\cap\SS^n_{++} = \emptyset$ then by \cite[Theorem 11.2]{rockafellar_convex_2015} there exists $\Xb\in\SS^n$ such that $\Xb\bullet\Yb > 0,\ \forall \Yb\in\SS^n_{++},$ and $\Xb\bullet\sum_{i=1}^{m}\gl_i\Ab_i = 0, \  \forall \ggl\in \R^m$. From the first condition we get that $\Xb\in\SS^n_{+}\setminus\lrbr{\Ob}$ by self-duality of $\SS_+^n$. The second one holds in particular for $\hat{\gl}_i = -\sign(\Ab_i\bullet\Xb),\ i\in\irg{1}{m}$ so that $ \sum_{i=1}^{m}\hat{\gl}_i\Ab_i\bullet\Xb = 0$ implies $\Ab_i\bullet \Xb  = 0$. But then there is a nonzero $\Xb\in\SS_+^n$ so that $\AA(\Xb,\oo) = \oo$ in contradiction to (\ref{eqn:Nec2}).  
	\end{proof}
	
	\begin{thm}
		The following instances of $\XX,\ \AA$, and $\b$ fulfill conditions \eqref{eqn:Nec1} and \eqref{eqn:Nec2}.
		\begin{align*}
			\XX
			&\coloneqq 
			\lrbr{\x\in\lrbr{0,1}^n\colon \Gb\x = \f }, \ \AA(\Xb,\x) \coloneqq 
			\begin{bmatrix}
				\diag(\Xb)-\x\\
				\Gb\x \\
				\diag(\Gb\Xb\Gb\T)
			\end{bmatrix},\ 
			\b \coloneqq 
			\begin{bmatrix}
				\oo \\ \f \\ \f\circ\f
			\end{bmatrix},\\
			\XX
			&\coloneqq 
			\lrbr{\x\in\R^n\colon\x\T\Ab\x + \a\T\x  = \ga},\
			\AA(\Xb,\x) \coloneqq  
			\begin{bmatrix}
				\Ab\bullet\Xb + \a\T\x
			\end{bmatrix},\ 
			\b \coloneqq 
			\begin{bmatrix}
				\ga 
			\end{bmatrix}, \mbox{ for } \Ab\in\SS_{++}^n.
		\end{align*}
	\end{thm}
	\begin{proof}
		Condition (\ref{eqn:Nec1}) is easily validated by checking $\AA(\x\x\T,\x) =\b$ for all $\x\in\XX$. Finally, condition (\ref{eqn:Nec2}) is guaranteed by \cref{lem:PSDsum} since $\Ab\in\SS^n_{++}$ and $\diag(\Xb)_i=\e_i\e_i\T\bullet\Xb$ and  $\sum_{i=1}^{n}\e_i\e_i\T = \Ib \in\SS^{n}_{++}$.
	\end{proof}
	
	\begin{rem}
		In \cref{sec:A primal heuristic} we will describe a third instance not discussed in the above proposition, which is, hence, not exhaustive but should convince the reader that interesting cases are covered, i.e., binary optimization and optimization over a spheroid.  
	\end{rem}
	Note, that in \cref{thm:ConcaveTentOverExtreme} we did not explicitly assume that $\XX$ is a subset of the extreme points of $\conv(\XX)$ or that it was bounded. However, (\ref{eqn:Nec1}) and (\ref{eqn:Nec2}) imply just that as demonstrated below, which establishes the connection between the conditions in \cref{thm:ExistenceofCT} and \cref{thm:ConcaveTentOverExtreme}.
	
	\begin{prop}\label{prop:BoundedXX} 
		If condition \eqref{eqn:Nec2} holds, then $\XX$ is bounded and if in addition \eqref{eqn:Nec1} holds, then $\XX$ is a subset of extreme points of $\conv(\XX)$. 
	\end{prop}
	\begin{proof} 
		Under \eqref{eqn:Nec2} if $\AA(.,.)$ is in the form \eqref{eqn:AARepresentation}, then by \cref{lem:PSDsum} there is $\ggl\in\R^m$ such that $\Ab\coloneqq\sum_{i=1}^{m}\gl_i\Ab_i\in\SS_{++}^n$. Define $\a \coloneqq \sum_{i=1}^{n}\gl_i \a_i$ and  $b~\coloneqq~ \ggl\T\b$. We have that  
		\begin{align*}
			\bar{\Ab}\coloneqq 
			\ga
			\begin{bmatrix}
				1 & \oo\T \\ \oo & \Ob
			\end{bmatrix}+ 
			\begin{bmatrix}
				-b & \sfrac{1}{2}\a\T \\
				\sfrac{1}{2}\a & \Ab
			\end{bmatrix}\in \SS_{++}^{n+1} 
			\quad 
			\Leftrightarrow
			\quad 
			\ga-b > 0,\ 
			\Ab + 
			\frac{1}{4(b-\ga)}\a\a\T \in \SS_{++}^n,
		\end{align*}
		by Schur complementation and since $\ga-b \rightarrow \infty$ and $|b-\ga|^{-1} \rightarrow 0 $ as $ \ga \rightarrow \infty$  and $\Ab\in \SS^n_{++} = \interior\SS_+^n$, there is an  $\ga>0$ such that these positive definiteness conditions hold. From this $\ga$, we construct the set $\GG_{\Ab} \coloneqq \lrbr{\Yb\in\SS_{+}^{n+1}\colon \bar{\Ab}\bullet \Yb = \ga}$ and we see that 
		\begin{align*}
			\hspace{-0.8cm}
			\GG_{\Ab}=
			\lrbr{
				\begin{bmatrix}
					x_0 & \x\T \\ \x & \Xb
				\end{bmatrix}\in\SS_+^{n+1}
				\colon 
				\begin{bmatrix}
					\ga-b & \sfrac{1}{2}\a\T \\
					\sfrac{1}{2}\a & \Ab
				\end{bmatrix}
				\bullet 
				\begin{bmatrix}
					x_0 & \x\T \\ \x & \Xb
				\end{bmatrix} = \ga 
			}
			\supseteq
			\lrbr{
				\begin{bmatrix}
					1 & \x\T \\ \x & \Xb
				\end{bmatrix}\in\SS_+^{n+1}
				\colon 
				\AA(\Xb,\x) = \b
			},
		\end{align*}
		by the construction of $\bar{\Ab}$, since the equation of $\GG_{\Ab}$ is a linear combination of all equations of the contained set (including $x_0=1$). From \eqref{eqn:Nec1} we conclude $\GG_{\Ab}\supseteq \GG(\XX)$. In addition, since $\bar{\Ab}\in\SS_{++}^{n+1}$ we see that $\GG_{\Ab}$ is bounded so that $\GG(\XX)$ is bounded too, in other words 
		\begin{align*}
			\left|
			\begin{vmatrix}
				1 & \x\T \\ \x & \Xb
			\end{vmatrix} 
			\right|^2_{F} 
			= 
			1 +2\x\T\x + \Xb\bullet\Xb
			\leq M,\quad \forall (\x,\Xb)\in\GG(\XX),
		\end{align*}
		for some big enough $M$ and since $1+\Xb\bullet\Xb>0$ and $(\x,\x\x\T)\in\GG(\XX)$ whenever $\x\in\XX$ we see that $\|\x\|_2^2\leq M/2,\ \forall \x\in\XX$. 
		
		For the second assertion, assume that $\x\in\XX\subseteq\conv(\XX)$. Then it is a convex combination of $\x_i\in\XX, \ i \in\irg{1}{k}$ with weights $\ggl\in\Delta_k$ (again, $k$ can be bounded above by Caratheodory's Theorem). We have that $(\x,\x\x\T)\in\GG(\XX)$ and $(\x_i,\x_i\x_i\T)\in\GG(\XX), \ i \in\irg{1}{k}$, hence for $\Xb\coloneqq \sum_{i=1}^{k}\gl_i\x_i\x_i\T$ also $(\x,\Xb)= \sum_{i=1}^{k}\gl_i(\x_i,\x_i\x_i\T)\in\GG(\XX)$ so that $\AA(\Xb,\x) = \b = \AA(\x\x\T,\x)$ by \eqref{eqn:Nec1}, which implies $\oo = \AA(\Xb,\x)-\AA(\x\x\T,\x) = \AA(\Xb-\x\x\T,\oo)$. From \eqref{eqn:Nec2} we thus get $\Xb=\x\x\T$ so that we arrive at  
		\begin{align*}
			\begin{bmatrix}
				1 & \x \T \\ \x & \x\x\T 
			\end{bmatrix}
			= \begin{bmatrix}
				1 & \x \T \\ \x & \Xb 
			\end{bmatrix}
			= 
			\sum_{i=1}^{k}
			\gl_i
			\begin{bmatrix}
				1\\ \x_i
			\end{bmatrix}
			\begin{bmatrix}
				1\\ \x_i
			\end{bmatrix}\T 
			\mbox{ which implies }
			\gl_i
			\begin{bmatrix}
				1\\ \x_i
			\end{bmatrix}
			\begin{bmatrix}
				1\\ \x_i
			\end{bmatrix}\T
			= 
			\ga_i 
			\begin{bmatrix}
				1\\ \x
			\end{bmatrix}
			\begin{bmatrix}
				1\\ \x
			\end{bmatrix}\T,
		\end{align*}
		for some $\ga_i\geq 0, \ i \in \irg{1}{k}$ by \cref{lem:ExtremeRank1} a). Since all $\gl_i>0$, as otherwise we could drop it from the convex combination, the final equation above implies that $\gl_i = \ga_i$ so that $\ga_i/\gl_i = 1$ and \cref{lem:ExtremeRank1}$~$b) gives $[1,\x_i\T]\T = \pm[1,\x\T]\T$. But $1\neq -1$ so the only solution is $\x = \x_i$ for all $i\in\irg{1}{k}$ so that $\x$ is indeed an extreme point of $\conv(\XX)$. 		
	\end{proof}
	
	\subsection{A double lifting approach for general $\XX$ }\label{sec:A double lifting approach}
	By \cref{thm:ExistenceofCT}, a concave tent may fail to exist if the set $\XX$ is not a set of extreme points of $\conv(\XX)$. In this section, we are going to discuss how to circumvent this issue by employing a double lifting strategy. The theorem below can be interpreted as a generalization of \eqref{thm:chlifting}. 
	
	\begin{thm}%\label{thm:chlifting}
		Let $\XX\subseteq \R^n$ be a compact and let $f$ be a function conforming to \cref{asm:CompactU}, where $h$ is equal to zero everywhere. Assume, the function $g\colon \R^n\times \SS^n\rightarrow \R$ is given by
		\begin{align*}
			g(\x,\Xb) 
			\coloneqq 
			\sup_{\Psi,\u,\Ub} \lrbr{
				\begin{array}{l}
					\Ab\bullet\Xb + 2\a\T\x + 2\Bb\T\bullet \Psi + \Cb\bullet\Ub + 2\cc\T\u 
					\colon
					\\
					\begin{bmatrix}
						1  & \u\T & \x\T \\
						\u & \Ub  & \Psi\T \\
						\x & \Psi & \Xb
					\end{bmatrix} \in \CC, \
					\begin{bmatrix}
						1  & \u\T \\
						\u & \Ub   
					\end{bmatrix}\in\GG(\UU),\ 			
				\end{array}				
			},
		\end{align*}
		where $\CC\subseteq \SS^{n+q+1}$ is as in \cref{thm:ConcaveTentOverExtreme}.
		Also, let $\bar{h}\colon \R^n\times \SS^n \rightarrow \R,\ \left(\x,\Xb\right)\mapsto \bar{h}(\x,\Xb)$ be a closed concave function. Then 
		\begin{align*}
			v^*
			&\coloneqq 
			\inf_{\x\in\R^n} \lrbr{\bar{h}(\x,\x\x\T)+f(\x)\colon \x\in\XX}=\inf_{\x\in\R^n,\, \Xb\in\SS^n}\lrbr{\bar{h}(\x,\Xb)+g(\x,\Xb)\colon 
				\begin{bmatrix}
					1 & \x\T \\ \x & \Xb
				\end{bmatrix}\in \GG(\XX)
			}.				
		\end{align*}
	\end{thm}
	\begin{proof}
		First, we observe that $g$ is concave by \cref{lem:Concavity} and it is closed function by similar arguments as in \cref{thm:ConcaveTentOverExtreme}, where boundedness follows again from $(\u,\Ub)$ being bounded on $\GG(\UU)$, $(\x,\Xb)$ is fixed and so $\Psi$ is bounded since $\CC\subseteq \SS_{+}^{n+q+1}$. In addition $g$ is finite on $\GG(\XX)$ since $\CC\supseteq \GG(\XX\times\UU)$. Define $\XX_{lift}\coloneqq \lrbr{(\x,\x\x\T)\in\SS^{n+1}\colon \x\in\XX}\subseteq \GG(\XX)$ and  $f_{lift} \colon \XX_{lift} \rightarrow \R, \ (\x,\x\x\T)\mapsto f(\x)$. Then $v^* = \inf_{\x,\Xb}{}\lrbr{\bar{h}(\x,\Xb)+f_{lift}(\x,\Xb) \colon (\x,\Xb)\in\XX_{lift}}$. We also see that $g(\x,\Xb)= f(\x) = f_{lift}(\x,\Xb)$ for all $(\x,\x\x\T)\in\XX_{lift}$  by repeating the arguments \cref{thm:ConcaveTentOverExtreme} (once $\Xb=\x\x\T$, which is always true on $\XX_{lift}$, the arguments proceed verbatim). Thus, $g$ is a concave tent of $f_{lift}$ over $\XX_{lift}$ so that $v^* = \inf_{\x,\Xb}{}\lrbr{\bar{h}(\x,\Xb)+g(\x,\Xb) \colon (\x,\Xb)\in\conv(\XX_{lift})}$ by \cref{thm:ConcaveReformulations} and since $g$ is closed, i.e. upper semicontinuous (see \cite[Theorem 7.1]{rockafellar_convex_2015}) we can replace the feasible set with its closure $\cl\conv(\XX_{lift})=\GG(\XX)$.
	\end{proof}
	
	We can think of $g$ as the result of a double lifting. First, we lift the $f$ into a higher dimensional space to obtain a function $f_{lift}$ defined on the extreme points of $\GG(\XX)$, for which we build, in a second lifting step, the function $g$, which is  a concave tent of $f_{lift}$ over $\GG(\XX)$. The necessary condition in \cref{thm:ExistenceofCT} is circumvented since even a non-extreme point $\x\in\conv(\XX)\cap\XX$ corresponds to a point $(\x,\x\x\T)$ in the lifted space that is an extreme point of $\GG(\XX)$. Also, the boundedness condition on $\XX$ is alleviated since in the proof of \cref{thm:ConcaveTentOverExtreme} it is merely needed (in the guise of \eqref{eqn:Nec2}) to ensure the variables $\Xb=\x\x\T$ whenever $\x\in\XX$, but this feature is absorbed by the structure of $\XX_{lift}$ in the theorem above. 
	
	The function $h$ is replaced by $\bar{h}$ which can have a more complicated structure since it only needs to be concave in the lifted space. For example, $\bar{h}$ can be the pointwise infimum of linear functions $\bar{h}(\x,\Xb) \coloneqq \inf_{i\in\II}\lrbr{\Qb_i\bullet\Xb+2\q_i\T\x+q_i}$ in the lifted space so that in under the restriction $\Xb = \x\x\T$ the function becomes the pointwise infimum of (not necessarily convex) quadratic functions in the original space of variables. If we set $\bar{h}$ to be a quadratic function and $f$ to be constant zero then the above theorem recovers \eqref{thm:chlifting}. Thus, our theory can be seen as an expansion of the classical copositive optimization framework discussed in the introduction.  
	
	It is well known that even linear optimization over $\GG(\XX)$ can be very difficult let alone nonlinear optimization. We will defer investigations on how to navigate this difficulty to future research. Therefore, our discussion henceforth will focus on the concave tents described in \cref{thm:ConcaveTentOverExtreme}. 
	
	\subsection{Some guidance on $\GG(\XX\times\UU)$ and its relaxations}
	In order to build the our concave tents an exact representation of $\GG(\XX\times\UU)$ is not necessary in principle, since any instance of $\CC$ (defined as in \cref{thm:ConcaveTentOverExtreme}) that is a subset of the positive semidefinite cone will suffice. However, the tighter we approximate $\GG(\XX\times\UU)$, the smaller the concave tent will be (since we restrict the feasible set of the defining supremum problem), which reduces the risk of creating spurious local minima in the concave reformulation. As stated in the introduction, the literature on the transformation $\GG(.)$ is vast, so many tools exist for tackling this challenging object. However, if this transformation is applied to the Cartesian product $\XX\times\UU$ the situation becomes even more complicated, despite the usefulness of akin objects, that appear in for example \cite{xu_copositive_2018, xu_improved_2023} to name a few.

	Characterizations of $\GG(\FF)$ for any arbitrary $\FF\subseteq \R^n$ often involve the so-called set-completely positive cone $\CPP(\KK) \coloneqq \cl\conv\lrbr{\x\x\colon \x\in\KK}\subseteq \SS^n$, where $\KK\subseteq\R^n$ is a closed convex cone, and are regularly given as 
	\begin{align}
		\GG(\FF) 
		\coloneqq \lrbr{
			(\x,\Xb)\in\CPP(\R_+\times\KK_{\FF}) \colon \AA_{\FF}(\x,\Xb) =\b_{\FF}
		},
	\end{align}
	where $\KK_{\FF},\ \AA_{\FF},\ \b_{\FF}$ are a suitably chosen closed convex cone, linear map, and vector, respectively (see \cite[Section 2.1]{bomze_optimization_2023} for a thorough introduction to such constructions). This construction shifts the difficulty into the conic constraint.  In general, certifying membership in such a cone is NP-hard. An extensive list of references to results on these cones and their approximations schemes can be found in \cite[Section 2.4]{bomze_optimization_2023}.

	If we have such characterizations of $\GG(\XX)$ and $\GG(\UU)$ than it is immeadiate that 
	\begin{align*}
		\GG(\XX\times\UU)
		=
		\lrbr{
			\Mb\in\CPP(\R_+ \times \KK_{\XX} \times \KK_{\UU})
			\colon 
			\AA_{\XX}(\x,\Xb) = \b_{\XX},\
			\AA_{\UU}(\u,\Ub) =\b_{\UU}				
		},
	\end{align*}
	with $\Mb\in\SS^{n+q+1}$ is a matrix with the same subdivisions as those in the description of \eqref{eqn:CT1}, which likewise shifts the difficulty to the characterization of $\CPP(\R_+\times\KK_{\XX}\times\KK_{\UU})$. In case both $\KK_{\XX}\times\KK_{\UU}$ are the nonnegative orthant, one can work with the approximation strategies mentioned above. However, outside this special case, this matrix cone is notoriously difficult to work with. Since $\CPP(\R_+\times\KK_{\XX}\times\KK_{\UU}) \implies (\x,\Xb)\in\CPP(\R_+\times\KK_{\XX}),\ (\u,\Ub)\in\GG(\R_+\times \KK_{\UU})$ a reasonable startingpoint for an outer approximation of $\GG(\XX\times\UU)$ is given by 
	\begin{align*}
		\CC_1 
		\coloneqq &
		\lrbr{
			\Mb\in\SS_+^{n+q+1}
			\colon 
			\begin{array}{l}
				\AA_{\XX}(\x,\Xb) = \b_{\XX},\ (\x,\Xb)\in\CPP(\R_+\times\KK_{\XX}),\\
				\AA_{\UU}(\u,\Ub) = \b_{\UU},\ (\u,\Ub)\in\CPP(\R_+\times \KK_{\UU})
			\end{array}
		}\\
		= &
		\lrbr{
			\Mb\in\SS_+^{n+q+1} 
			\colon  
			(\x,\Xb)\in\GG(\XX),\ 
			(\u,\Ub)\in\GG(\UU)
		},			
	\end{align*}
	which is still not tractable, due to the set-completely positive constraint. Still we have some takeaways from these two representations. The first tells us that even if we have exact representations of of $\CPP(\R_+\times\KK_{\XX})$ and $\CPP(\R_+\times\KK_{\UU})$ we cannot exactly represent $\CPP(\R_+\times\KK_{\XX}\times\KK_{\UU})$, but all knowledge on the former two cones can be used to strengthen relaxations of $\CC_1$. Also, a strengthening can be obtained by adding valid constraints on $\Psi$. For example, if $\KK_X\subseteq \R_+^{n_x}$, then $\mathrm{Rows}(\Psi)\subseteq\KK_{\UU}$. Later, in \cref{sec:Heuristic}, we will use results from \cite{noauthor_zhen_nodate} for generating further valid constraints on $\Psi$. Other ways to tighten this representation are discussed in \cite[Section 2.3.]{bomze_optimization_2023} and references therein. The second representation of $\CC_1$ show us that, likewise, all knowledge on $\GG(\XX)$ and $\GG(\UU)$ can be used to strengthen its relaxations; again, we refer to \cite[Section 2.1]{bomze_optimization_2023} and references therein. 
	
	\section{Quality of our concave tents and comparisons}
	We will now investigate the quality of the concave tents constructed above in the sense that we will compare them to two important classical instances of concave tents: the concave envelope and concave reformulations based on concave quadratic updates. We will show that, under some technical assumptions, our concave tents are sandwiched between the two and that they attain the lower bound, which is the concave envelope. 
	
	\begin{rem}
		In %Theorem 
		\cref{thm:ConcaveEnvelope} and \ref{thm:CTvsClassical} below we assume that the concave term $h$ in the definition of $f$ is removed. We conjecture that similar statements can be proved with $h$ included, most likely after a thorough modification of the definition of $g$. However, the arguments, if valid at all, would be significantly more complicated. To keep the present text concise, we have therefore decided to defer such proofs to a future short communication in case our conjecture turns out to be true.  
	\end{rem}

	\subsection{Concave envelope characterization}
	Based on recent results in \cite{gabl_finding_2024} we can show that in case $\CC=\GG(\UU\times\XX)$, a respective concave tent constructed in \cref{thm:ConcaveTentOverExtreme} is actually the concave envelope of $f$ over $\XX$. Thus, our family of concave tents includes the best concave tent in the sense that it coincides with the smallest concave overestimator of $f$ over $\XX$. Larger concave tents may have an increased number of local minima which is computationally undesirable, so the result below is encouraging.     
	
	\begin{thm}\label{thm:ConcaveEnvelope}
		Maintain the assumptions of \cref{thm:ConcaveTentOverExtreme} and assume that $h$ is the zero function, $\XX$ is closed, and $\CC=\GG(\UU\times\XX)$. Then the function $g$ is the concave envelope of $f$ over $\XX$. 
	\end{thm}
	\begin{proof} 
		Denote the concave envelope of $f$ over $\XX$ as $\hat{f}$.
		Note, that by \eqref{eqn:Nec1} the linear constraint can be omitted in the description of $g$ given the present choice of $\CC$, so that $g$ is of a form such that it falls under the regime of \cite[Theorem 6.]{gabl_finding_2024}. In the language of that theorem, the triple $(f,g,\UU\times\XX)$ takes the role of $(\phi,\hat{\phi},\FF)$ discussed there.
		We already argued in \cref{thm:ConcaveTentOverExtreme} that the supremum defining $g$ has zero duality gap. This implies by point 5.\ of \cite[Theorem 6.]{gabl_finding_2024} that $g=\cl(\hat{f})$. However, from the proof of \cref{prop:continuity} we know that $f$ is continuous if $h$ is continuous, which is the case for the zero function. The concave envelope of continuous functions over compact sets is itself closed by \cite[Corollary 17.2.1.]{rockafellar_convex_2015}. Thus, $g = \cl(\hat{f}) = \hat{f}$ as desired. 
	\end{proof}
	
	While the theorem states that it is in principle possible to recover the concave envelope, the assumptions are much more restrictive than what is needed for the derivation of weaker concave tents, since the set $\GG(\UU\times\XX)$ is highly intractable for many instances of $\XX$ and $\UU$. Thus, a trade-off has to be made between the quality of the concave tent and the computational effort caused by strengthening the cone $\CC$. In \cref{sec:Heuristic} we give an example based on semidefinite and second-order cone constraints that performs reasonably well in the numerical experiments discussed later in \cref{sec:Numerical}.
	
	\subsection{Comparison to a classic construction}
	For $\XX\coloneqq \lrbr{\x\in\R^n\colon \AA(\x\x\T,\x) = \b}$, where $\AA$ is of the form \eqref{eqn:AARepresentation} and conforms with \eqref{eqn:Nec2} and for a function $f$ that conforms with \cref{asm:CompactU} there is a classical way to construct concave tents by considering 
	\begin{align}\label{eqn:classicCT}
		g_{c}(\x,\ggl) \coloneqq f(\x)-\ggl\T(\AA(\x\x\T,\x)-\b) =f(\x)- \sum_{i=1}^{m}\gl_i(\Ab_i\bullet\x\x\T+\a_i\T\x-b_i).
	\end{align}
	Whenever $\x\in\XX$ we have that $g_{c}(\x,\ggl) = f(\x)$, and by \cref{lem:PSDsum} there is a choice for $\ggl$ such that $-\ggl\T\AA(\x\x\T,\x)$ becomes a concave quadratic function, which opens up the possibility that $g_{c}(\x,\ggl)$ becomes concave in $\x$ (perhaps after a scaling of $\ggl$) and therefore a concave tent of $f$ over $\XX$. Such approaches have for example been used in \cite{pardalos_constrained_2006,raghavachari_connections_1969} as discussed in the introduction. However, outside of such special cases, it is nontrivial to find the required value for $\ggl$ if such values exist at all.  If the Hessian of $f$, say $\Hb_f$, exists for all $\x\in\XX$, then it suffices to choose $\ggl$ such that the hessian of $g_c$ fulfills $-(\Hb_{f}(\x)-\sum_{i=1}^{m}2\gl_i\Ab_i )\in\SS_+^n,\ \forall \x\in\XX$, which may  be an intractable choice unless $\Hb_{f}(\x)$ depends on $\x$ in a very simple manner. In addition, $\Hb_{f}$ might be difficult to find or be nonexistent. It may be even more difficult to find the value of $\ggl$ that gives the smallest concave tent based on $g_c$. 
	Finally, even if an appropriate choice of $\ggl$ does exist, finding supergradients of $g_{c}$ may be challenging, for example, if the supremum term in $f$ does not allow for a simple construction of derivatives. By comparison, our approach allows for super-differentiation under mild assumptions as will be discussed in \cref{sec:Super-derivatives based on conic duality}.    
	
	On the other hand, if an appropriate choice for $\ggl$ does exist and the resulting concave tent can be differentiated easily, then the classical construction could be computationally much cheaper than the ones we proposed in the prior sections. Still, this hypothetical advantage is paid for in terms of the quality of the so-constructed concave tent. Indeed, in the main result of this section, we will show that even the weakest of the concave tents introduced in \cref{thm:ConcaveTentOverExtreme} underestimates any of the concave tents based on $g_{c}$.
	
	\begin{lem}\label{lem:GGUUdecomp}
		Let $\R^q\supseteq\UU\neq\emptyset$ be compact. Then for $\Ub\in\SS^{q+1},\ \Psi\in\R^{n\times (q+1)}, \ \Xb\in\SS^n$ we have 
		\begin{align*}
			\begin{bmatrix}
				\Ub & \Psi\T \\ \Psi & \Xb
			\end{bmatrix}\in\GG(\UU\times\R^n),
			\quad 
			\mbox{ if and only if }
			\quad 
			\begin{bmatrix}
				\Ub & \Psi\T \\ \Psi & \Xb
			\end{bmatrix}\in\SS^{q+n+1},\ \Ub\in\GG(\UU).
		\end{align*}	
	\end{lem}
	\begin{proof}
		If $\Mb\in\GG(\UU\times\R^n)$ then it is also psd since matrices of the form $\y\y\T$ are psd and the psd-cone is closed. Further, it is the limit point of $\Mb_j = \sum_{i=1}^{k}\gl_i\y^j_i(\y^j_i)\T$, with $ \y^j_i = [1,(\u^j_i)\T,(\x^j_i)\T]\in\lrbr{1}\times\UU\times\R^n,\ i\in\irg{1}{k},\ \ggl^j\in\Delta_k,\ \forall j\in\N$ so that the $\Ub$-block of $\Mb$ is the limit point of convex combinations of rank one matrices recruited from $\lrbr{1}\times\UU$, which certifies $\Ub\in\GG(\UU)$. 
		
		Conversely, from the psd constraint we get a factorization $\Ub = \Vb\Vb\T,\ \Psi = \Vb\Zb\T, \ \Xb = \Zb\Zb\T$ from some $\Vb\in\R^{q+1\times k}, \ \Zb\in\R^{n\times k}$. Also, since $\UU$ is compact we get from $\Ub\in\GG(\UU)$ that $\Ub = \sum_{i=1}^{l}\gl_i\w_i\w_i\T= \Wb\Diag(\ggl)\Wb$  for $\w_i\in\lrbr{1}\times\UU, \ i\in\irg{1}{l},\ \ggl\in\Delta_l$ and $\Wb\coloneqq [\w_1,\dots,\w_l]$. From Caratheodory's Theorem, we can always set $k = q+n+1$ and $l=q+1<k$, so that w.l.o.g. we can assume $k=l$ since if $l$ is too small we can split summands in the convex combinations until $l$ is large enough. We can also assume $\ggl>0$ since otherwise we can just drop the summands with zero weight. Therefore $\Lambda \coloneqq \Diag(\ggl)$, the diagonal matrix whose diagonal entries are the entries of $\ggl$, has an invertible square root.  Set $\ol{\Wb}\coloneqq \Wb\Lambda^{1/2}$, then the fact that $\ol{\Wb}\ol{\Wb}\T = \V\V\T$ implies that $\ol{\Wb}=\Vb\Qb$ for some orthonormal $\Qb\in\R^{k\times k}$ by \cite[Lemma 2.6.]{groetzner_factorization_2020}. Set $\Yb\coloneqq\Zb\Qb\Lambda^{\sfrac{-1}{2}}$, then we have 
		\begin{align*}				
			%\hspace{-1.2cm}
			\begin{array}{l}
				\Ub = \Vb\Qb\Qb\T\Vb\T = \ol\Wb\ol\Wb\T = \Wb\Lambda\Wb\T,\\
				\Psi = \Zb\Qb\Qb\T\Vb\T  = \Zb\Qb\ol\Wb\T = \Zb\Qb\Lambda^{\sfrac{-1}{2}}\Lambda\Wb\T = \Yb\Lambda\Wb\T, \\
				\Xb  = \Zb\Qb\Qb\T\Zb\T = 
				\Zb\Qb\Lambda^{\sfrac{-1}{2}}\Lambda\Lambda^{\sfrac{-1}{2}}\Qb\T\Zb\T=  	\Yb\Lambda\Yb\T,			
			\end{array}	
			\mbox{ so that }							
			\begin{bmatrix}
				\Ub & \Psi\T \\ \Psi & \Xb
			\end{bmatrix}
			=				
			\sum_{i=1}^{k}\gl_i
			\begin{bmatrix}
				\w_i \\ \y_i
			\end{bmatrix}
			\begin{bmatrix}
				\w_i \\ \y_i
			\end{bmatrix}\T \hspace{-0.3cm}, 
		\end{align*}
		where $\y_i$ is the $i$-th column of $\Yb$ and $[\w_i\T,\y_i\T]\T\in\lrbr{1}\times\UU\times \R^n, \ i\in\irg{1}{k}$, so that the matrix-block is indeed an element of $\GG(\UU\times\R^n)$.
	\end{proof}

	\begin{thm}\label{thm:CTvsClassical}
		Let $\XX\coloneqq \lrbr{\x\in\R^n\colon \AA(\x\x\T,\x) = \b}$ such that $\AA$ is of the form \eqref{eqn:AARepresentation} and fulfills \eqref{eqn:Nec2} and let $f$ fulfill \cref{asm:CompactU} with $h$ being the zero function. Let $g$ be as in \cref{thm:ConcaveTentOverExtreme} and define $\TT\coloneqq\lrbr{\ggl\in\R^m\colon g_{c}(.,\ggl) \mbox{ is concave}}$. Then $g(\x)\leq g_{c}(\x,\ggl),\ \forall \x\in\R^n,\  \forall \ggl\in\TT$.
	\end{thm}
	\begin{proof}
		Since $g$ becomes larger with $\CC$ it suffices to discuss the case $\CC=\SS_+^{n+q+1}$.
		If $\TT =\emptyset$ the statement is trivially true, so assume otherwise that $\bar{\ggl}\in\TT$. By \cref{prop:continuity}, $f$ is upper semicontinuous and nowhere infinity so that $g_c(\,.\,,\bar{\ggl})$ is closed proper concave function. Thus, by \cite[Theorem 12.1]{rockafellar_convex_2015} we have $g_c(\,.\,,\bar{\ggl}) = \inf_{[\v\T,v_0]\T\in\VV}\lrbr{\v\T\x+v_0}$, for some set $\VV\subseteq \R^{n+1}$. We will now show that in fact $g_c(\x,\bar{\ggl}) = \hat{g}_c(\x,\bar{\ggl})$ where 
		\begin{align*}
			\hat{g}_c(\x,\bar{\ggl}) \coloneqq \sup_{\Psi,\u,\Ub,\Xb} \lrbr{
				\begin{array}{l}
					\Ab\bullet\Xb+2\a\T\x+2\Bb\T\bullet\Psi+\Cb\bullet\Ub+2\cc\T\u - \bar{\ggl}\T\left(\AA(\Xb,\x)-\b\right)
					\colon
					\\
					\begin{bmatrix}
						1  & \u\T & \x\T \\
						\u & \Ub  & \Psi\T \\
						\x & \Psi & \Xb
					\end{bmatrix} \in \SS_+^{n+q+1}, \
					\begin{bmatrix}
						1  & \u\T \\
						\u & \Ub   
					\end{bmatrix}\in\GG(\UU),
				\end{array}				
			}.
		\end{align*}
		We have "$\leq$" since for any $\x$ we have a solution $\u\in\UU$ to the inner supremum defining $f$ and $\Xb=\x\x\T, \ \Psi= \x\u\T, \ \Ub = \u\u\T$ is a feasible solution to the supremum defining $\hat{g}_c$.  For the converse, consider that  $g_c(\x,\bar{\ggl})= \inf_{[\v\T,v_0]\T\in\VV}\lrbr{\v\T\x+v_0}$ implies
		\begin{align*}
			& \
			g_c(\x,\bar{\ggl})
			\leq
			\v\T\x+v_0,
			\quad 
			\forall \x\in\R^n,\ [\v\T,v_0]\T\in\VV,\\
			\Leftrightarrow
			&
			\sup_{\x\in\R^n}
			\lrbr{
				g_c(\x,\bar{\ggl})-\v\T\x
			}
			\leq
			v_0,
			\quad 
			\forall 
			[\v\T,v_0]\T\in\VV,
		\end{align*}
		and by \eqref{thm:chlifting} we can reformulate 
		\begin{align*}
			\hspace{-1cm}
			\sup_{\x\in\R^n}\lrbr{g_c(\x,\bar{\ggl})-\v\T\x} 
			&
			= 
			\sup_{\x,\Psi,\u,\Ub,\Xb} \lrbr{
				\begin{array}{l}
					\Ab\bullet\Xb+2\a\T\x+2\Bb\T\bullet\Psi+\Cb\bullet\Ub+2\cc\T\u-\bar{\ggl}\T\AA(\Xb,\x)-\v\T\x
					\colon
					\\
					\begin{bmatrix}
						1  & \u\T & \x\T \\
						\u & \Ub  & \Psi\T \\
						\x & \Psi & \Xb
					\end{bmatrix} \in \GG(\UU\times\R^n), \
				\end{array}				
			}\\
			&
			=
			\sup_{\x\in\R^n}\lrbr{\hat{g}_c(\x,\bar{\ggl})-\v\T\x},
		\end{align*}
		where the second equality is valid since the matrix constraint involving $\GG(\UU\times\R^n)$ can be replaced by the matrix constraints in the definition of $\hat{g}_c(\x,\bar{\ggl})$ by \cref{lem:GGUUdecomp}. Thus, we get 
		\begin{align*}
			&
			\sup_{\x\in\R^n}
			\lrbr{
				\hat{g}_c(\x,\bar{\ggl})-\v\T\x
			} = 
			\sup_{\x\in\R^n}
			\lrbr{
				g_c(\x,\bar{\ggl})-\v\T\x
			}\leq v_0, \quad 
			\forall 
			[\v\T,v_0]\T\in\VV,\\
			\Leftrightarrow
			& \
			\hat{g}_c(\x,\bar{\ggl})
			\leq
			\v\T\x+v_0,
			\quad 
			\forall \x\in\R^n,\ [\v\T,v_0]\T\in\VV,
		\end{align*}
		so that indeed $\hat{g}_c(\x,\bar{\ggl})\leq \inf_{(\v\T,v_0)\T\in\VV}\lrbr{\v\T\x+v_0} = g_c(\x,\bar{\ggl})$. In addition we can also see that the supremum defining $\hat{g}_c$ is unbounded if there is an $\Xb\in\SS_+^n$ such that $\Ab\bullet\Xb-\sum_{i=1}^{m}(\bar{\ggl})_i\Ab_i\bullet\Xb>0$, as this $\Xb$ would span an improving ray. But since $\hat{g}_c\leq g_c<\infty$ as $g_c$ is proper, we see that our assumptions imply $-(\Ab-\sum_{i=1}^{m}(\bar{\ggl})_i\Ab_i)\in\SS_+^n$. Thus, for any $\x\in\R^n$, we get 
		\begin{align*}
			\hspace{-0.5cm}		
			g_c(\x,\bar{\ggl})
			\geq &
			\inf_{\ggl\in\R^m}
			\lrbr{
				g_c(\x,\ggl)
				\colon 
				-(\Ab-\sum_{i=1}^{m}\gl_i\Ab_i)\in\SS_+^n
			}\\
			=
			&
			\inf_{\ggl\in\R^m}
			\lrbr{
				\sup_{\Psi,\u,\Ub,\Xb}\lrbr{\Ab\bullet\Xb+\dots}
				\colon 
				-(\Ab-\sum_{i=1}^{m}\gl_i\Ab_i)\in\SS_+^n
			}\\
			\geq
			&
			\sup_{\Psi,\u,\Ub,\Xb}
			\lrbr{
				\Ab\bullet\Xb+\dots
				+\inf_{\ggl\in\R^m}
				\lrbr{
					-\ggl\T
					\left(\AA(\Xb,\x)-\b\right)
					\colon 
					\sum_{i=1}^{m}\gl_i\Ab_i-\Ab\in\SS_+^n
				}
				\colon 
				\dots
			},
		\end{align*}  
		where the first inequality is due to the feasibility of $\bar\ggl$ in the right-hand side infimum. The next equation is due to $g_c = \hat{g}_c$ and the definition of the latter.  Finally, the second inequality is due to the min-max inequality. 
		The inner infimum in the last problem can, by weak conic duality, be underestimated via its dual given by $\sup_{\Yb\in\SS_+^n}\lrbr{\Ab\bullet\Yb\colon \AA(\Xb+\Yb,\x)=\b }$, which we can plug in to obtain the underestimator
		\begin{align*}
			g_c(\x,\bar{\ggl})
			&
			\geq
			\sup_{\Psi,\u,\Ub,\Xb,\Yb}
			\lrbr{
				\Ab\bullet(\Xb+\Yb)+\dots
				\colon 
				\AA(\Xb+\Yb,\x) = \b, \ \Yb\in\SS_+^n,\ 
				\dots
			}\\
			&
			\geq 
			\sup_{\Psi,\u,\Ub,\Xb}
			\lrbr{
				\Ab\bullet\Xb+\dots
				\colon 
				\AA(\Xb,\x) = \b, \
				\dots
			}=g(\x),
		\end{align*}
		where the second inequality stems from the fact that we effectively forced the constraint $\Yb =\Ob$, thus restricting the feasible set. In total we get $g(\x)\leq g_{c}(\x,\bar{\ggl})$ as desired.  
	\end{proof}

	\section{Supergradients based on conic duality}\label{sec:Super-derivatives based on conic duality}
	A powerful feature of the convex tents we introduced in previous sections is that obtaining a supergradient at a certain point comes for free with evaluating the convex tent at that point by conic duality theory, although there are some caveats to this statement. In this section, we will discuss this matter in greater detail. First, let us introduce the conic duals of the concave tents we have discussed so far explicitly.
	
	For convenience of notation, we define the closed convex cone 
	\begin{align*}
		\bar{\CC}
		\coloneqq 
		\lrbr{
			\begin{bmatrix}
				x_0  & \u\T & \x\T \\
				\u & \Ub  & \Psi\T \\
				\x & \Psi & \Xb
			\end{bmatrix} \in \mathrm{cone}\left(\CC\right)\colon \
			\begin{bmatrix}
				x_0  & \u\T \\
				\u & \Ub   
			\end{bmatrix}\in\mathrm{cone}(\GG_{\UU})\
		}\subseteq \SS^{n+q+1}_+,
	\end{align*}
	so that the different versions of $g$ discussed before can be cast as a linear conic optimization with a single conic constraint involving $\bar{\CC}$. In addition we need to define two linear maps $\AA_x\colon \R^n\rightarrow\R^m$ and $\AA_X\colon\SS^n\rightarrow\R^m$ such that $\AA(\Xb,\x) = \AA_X(\Xb)+\AA_x(\x)$ and we denote their adjoints by $\AA_X^*$ and $\AA_x^*$ respectively. 
	Replacing the supremum in $g$ with its dual yields its dual characterization:
	\begin{align*}
		g(\x)
		=  h(\x)+
		\inf_{
			\substack{\w\in\R^n,\, \ga\in\R,\\ \ggl\in\R^m}
		}
		\lrbr{
			\ga+\b\T\ggl+(2\a+2\w-\AA_x^*(\ggl))\T\x  
			\colon 
			\begin{bmatrix}
				\ga  & -\cc\T & \w\T \\
				-\cc & -\Cb  & -\Bb\T \\
				\w & -\Bb & \AA_X^*(\ggl)-\Ab
			\end{bmatrix}\in\bar{\CC}^*
		}
	\end{align*}
	where $\bar{\CC}^*\coloneqq \lrbr{\Mb\in\SS^{n+q+1}\colon \Mb\bullet\Yb \geq 0,\ \forall \Yb\in\bar{\CC}}$ denotes the dual cone of $\CC$. 
	Note that under the assumptions of \cref{thm:ConcaveTentOverExtreme} the feasible set of the supremum of the primal characterization of $g$ is always compact so that it enjoys zero duality gap by \cite[Proposition 2.8]{shapiro_duality_2001} as discussed in the proof of \cref{thm:ConcaveTentOverExtreme}. The following theorem shows that if we can solve the dual to $\epsilon$-optimality, we obtain an $\epsilon$-supergradient from the respective coefficients of $\x$ in the dual representation. We again restrict ourselves to the case where $h$ is the zero function. Otherwise, one can use the theorem in combination with standard supergradient calculus (specifically linearity of the supergradient map, see \cite{rockafellar_convex_2015}). We avoided such a consideration in our theorem for simplicity's sake. 
	
	\begin{thm}\label{thm:EpsilonSuperDerivatives}
		Consider the primal and dual characterizations of $g$ and assume that $h$ is the zero function. Let $\bar{\x}\in\dom(g)$ and let $(\w,\ga,\ggl)$ be a feasible solution for the dual characterization of $g(\bar{\x})$. Then $\bar{\y}\coloneqq (2\a+2\w-\AA_x^*(\ggl)) \in\partial_{\epsilon}\g(\bar{\x})$, the set of $\epsilon$-supergradients of $g$ at $\bar{\x}$, with $\epsilon = \ga+\b\T\ggl+\bar{\x}\T\bar{\y} - g(\bar{\x})\geq 0$, i.e. the optimality gap.
	\end{thm}
	\begin{proof}
		For an arbitrary $\x\in\R^n$ we have
		$
		g(\x)\leq \ga+\b\T\ggl+\x\T\bar{\y}
		$
		by weak conic duality and dual feasibility of $(\w,\ga,\ggl)$, which is independent of $\x$. Then, the following holds:
		\begin{align*}
			g(\x) & \leq \x\T\bar{\y} +\ga+\b\T\ggl 
			= \x\T\bar{\y} +\ga+\b\T\ggl +g(\bar{\x})-g(\bar{\x}) + \bar{\y}\T\bar{\x}-\bar{\y}\T\bar{\x}\\
			&= g(\bar{\x}) + \bar{\y}\T(\x-\bar{\x}) + \bar{\y}\T\bar{\x} +\ga+\b\T\ggl - g(\bar{\x})
			= g(\bar{\x}) + \bar{\y}\T(\x-\bar{\x}) + \epsilon, \ \forall \x\in\R^n.
		\end{align*}
	\end{proof} 
	
	There are open questions regarding $\epsilon$-supergradients obtained by solving the dual problem. Neither is the optimal solution necessarily unique nor is $\partial_{\epsilon}\g(\bar{\x})$ always a singleton and neither is the case for $\epsilon\neq 0$. It is unclear how the multiplicity of $\epsilon$-optimal solutions relates to the multiplicity of $\epsilon$-supergradients and, more importantly, there is no straightforward method of obtaining multiple solutions from a conic problem, if they exist. We will not investigate this ordeal any further here as it is a challenging one. Hopefully, we will be able to pursue it in future research.  
	
	\begin{rem}
		We gave a presentation of the dual in relatively abstract terms and although we will discuss an application of our theory in the next section, the dual characterization of the concave tent employed there will not be stated explicitly. For readers who want to learn the mechanics of deriving duals, we refer to \cite{roos_universal_2020}. A small example of such a dual was already provided as \eqref{eqn:SmallDual} presented in \cref{exmp:ConcaveTents}. In addition, we would like to point out that modern conic solvers such as Mosek report the dual variables associated with primal constraints automatically. Thus, a dual solution can be extracted from the primal formulation without the need to dualize it explicitly. We found that not only is this more convenient, but in our experiments eliciting the dual values in this manner was numerically much more stable, as working with the dual directly often led to Mosek reporting numerical problems.  
	\end{rem}

	\section{A primal heuristic for robust quadratic discrete optimization}\label{sec:A primal heuristic}
	In this section, we will apply concave tents to a certain type of robust quadratic binary problems with quadratic index, to derive a primal heuristic for this problem. We will first introduce the problem and derive a finite convex relaxation. The latter will typically yield fractional solutions and we will discuss a classical rounding scheme for obtaining close feasible solutions. As an alternative to this scheme, we then introduce a primal heuristic that exploits concave tents. In the succeeding section, we compare the performance of both "rounding" methods in a B\&B framework, where they are used to produce upper bounds the quality of which impacts the computational effort of the B\&B approach. 
	
	\subsection{Robust quadratic optimization over discrete sets}
	The problem we are interested in is given by 
	\begin{align}\label{eqn:Robust01QP}
		\begin{split}					
			v^* &\coloneqq 
			\min_{\x\in\XX}					
			\lrbr{		
				f(\x)\coloneqq 	
				\x\T\Ab\x+2\a\T\x+
				\sup_{\u\in\UU}
				\lrbr{
					2\u\T\Bb\x+\u\T\Cb\u+2\cc\T\u
				} 
			},
			\\
			\XX &\coloneqq
			\lrbr{\x\in\lrbr{-1,1}^n\colon u\geq \e\T\x\geq l},
		\end{split}
	\end{align}	
	which are discussed for example in \cite{bomze_interplay_2021,gokalp_robust_2017,xu_improved_2023}. They deal with optimizing over a set of binary vectors a quadratic function whose linear and constant parts are uncertain and depend on an uncertainty vector $\u$ that lives in an uncertainty set $\UU\subseteq \R^q$. We will focus on the case where the uncertainty set is a ball, i.e.
	\begin{align*}
		\UU \coloneqq \lrbr{\u\in\R^q\colon \|\u\|_2^2\leq 1}, 
		\mbox{ so that } 
		\GG(\UU) = 
		\lrbr{
			\begin{bmatrix}
				1 & \u\T \\ \u & \Ub
			\end{bmatrix}\in\SS_+^{q+1}
			\colon 
			\Ib\bullet\Ub \leq 1 
		},
	\end{align*}
	by the S-Lemma (see \cite{yakubovich_s-procedure_1971} or more closely related \cite{burer_gentle_2015}, also  \cite{xu_improved_2023} discuss this and other uncertainty sets  $\UU$ for which $\GG(\UU)$ has an exact tractable representation). Using classical techniques from \cite{ben-tal_adjustable_2004} (more generally discussed in \cite{bomze_interplay_2021}) we can reformulate the objective as  
	\begin{align}\label{eqn:fisanSDP}
		f(\x)&=	    		
		\x\T\Ab\x+2\a\T\x+
		\sup_{(\u,\Ub)\in\GG(\UU)}
		\lrbr{
			\Cb\bullet\Ub+2(\Bb\x+\cc)\T\u
		}\nonumber\\
		&=   		
		\x\T\Ab\x+2\a\T\x+
		\sup_{\u\in\R^q,\,\Ub\in\SS^{q+1}}
		\lrbr{
			\Cb\bullet\Ub+2(\Bb\x+\cc)\T\u
			\colon 
			\Ib\bullet\Ub\leq 1,\
			\begin{bmatrix}
				1&\u\T\\ \u & \Ub
			\end{bmatrix}\in\SS_+^{q+1}
		}\nonumber\\
		&=   		
		\x\T\Ab\x+2\a\T\x+
		\inf_{\ga\in\R,\, \gl\in\R}
		\lrbr{
			\ga-\gl
			\colon 
			\begin{bmatrix}
				\ga & (\Bb\x+\cc)\T\\
				\Bb\x+\cc & -\Cb-\gl\Ib
			\end{bmatrix}\in\SS^{q+1}_+, \ 
			\gl \leq 0 
		}, 
	\end{align}
	so that $f$ can be evaluated by solving a positive semidefinite optimization problem. We can plug this characterization into \eqref{eqn:Robust01QP} which yields,
	\begin{align}\label{eqn:Robust01QPSDP}
		\inf_{\x,\gl,\ga}
		\lrbr{
			\x\T\Ab\x+2\a\T\x+
			\ga-\gl			 
			\colon 
			\begin{array}{l}
				u\geq \e\T\x\geq l, \ 
				\x\in\lrbr{-1,1}^n,\\
				\begin{bmatrix}
					\ga & (\Bb\x+\cc)\T\\
					\Bb\x+\cc & -\Cb-\gl\Ib
				\end{bmatrix}\in\SS^{q+1}_+, \ 
				\gl \leq 0 
			\end{array}			
		},
	\end{align}
	hence, a binary SDP that can be tackled for example via B\&B discussed for example in \cite{gally_framework_2018}. 
	They rely on lower bounds based on convex relaxations which in our case can be constructed as follows:
	\begin{align}\label{eqn:RobustSemiRelaxation}
		\min_{\x,\Xb,\gl,\ga}
		\lrbr{
			\Ab\bullet\Xb+2\a\T\x+\ga-\gl 
			\colon 
			\begin{bmatrix}
				\ga & (\Bb\x+\cc)\T\\
				\Bb\x+\cc & -\Cb-\gl\Ib
			\end{bmatrix}\in\SS^{q+1}_+,\ \gl \leq 0,\
			(\x,\Xb)\in\GG(\XX)
		},
	\end{align}
	which is still intractable since it involves the intractable set $\GG(\XX)$. We propose the following approximation inspired by the discussion in \cite{burer_copositive_2012}. First, we elongate $\x$ by slack variables $\s\in\R^{2}_+$ to obtain $\z\coloneqq [\x\T,\s\T]\T$ and defining $\e_u\coloneqq [\e\T,1,0]$ and $\e_l\coloneqq [\e\T,0,-1]$ we can replace the inequalities by equations $\e_u\T\z = u, \ \e_l\T\z = l$. Also, define $\II_1 \coloneqq \irg{1}{n}$ and $\II_2 \coloneqq \irg{n+1}{n+2}$ and we see that 
	\begin{align*}
		\hspace{-1cm}
		\GG_{\XX}
		\coloneqq  
		\lrbr{
			(\x,\Xb)\in\R^n\times \SS^n 
			\colon 
			\begin{bmatrix}
				1 & \z\T \\ \z & \Zb
			\end{bmatrix}\in\SS_+^{n+3},
			\begin{array}{l}
				\e_u\T\z = u, \ \e_l\T\z = l,\
				\e_u\e_u\T\bullet\Zb = u^2,\ \e_l\e_l\T\bullet\Zb = l^2,\\
				%\diag(\Xb) = \e, \\ 
				(\z)_i=(\x)_i,\ i\in\II_1,\ (\Zb)_{ij}=(\Xb)_{ij}, \ (i,j)\in\II_1^2,\\
				(\z)_i\geq0,\ i\in\II_2,\ (\Zb)_{ij}\geq0, \ (i,j)\in\II_2^2,\ \diag(\Xb) = \e
			\end{array}
		},
	\end{align*}
	gives a tractable outer approximation of $\GG(\XX)$ (just check that a $\z$ associated with an  $\x\in\XX$ gives a $\Zb = \z\z\T$ associated with an $(\x,\Xb)\in\GG(\XX)\cap\GG_{\XX}$, the containment then follows from convexity and the definition of $\GG(\XX)$). Replacing $\GG(\XX)$ by $\GG_{\XX}$ gives a positive semidefinite relaxation that can be solved in polynomial time. Henceforth, we will refer to this relaxation as \eqref{eqn:RobustSemiRelaxation} for convenience. 
	
	The vector $\x$ in a solution of this SDP is not necessarily binary and it is desirable to have a rounding procedure that produces a feasible solution to upper bound the original problem. Both, good upper bounds and lower bounds are of great value for B\&B-based machinery since good quality upper bounds allow "bad" branches of the B\&B tree to be pruned quickly (see \cite[Section 7.3.]{gally_framework_2018}). An easy way to generate an upper bound from a nonbinary solution $\x_0$ to \eqref{eqn:RobustSemiRelaxation} is to look for the closest feasible solution by solving
	\begin{align}\label{eqn:RobustClosestFeasible}
		\min_{\x\in\R^n}
		\lrbr{
			\|\x_0-\x\|^2_2
			\colon 
			u\geq \e\T\x\geq l, \ 
			\x\in\lrbr{-1,1}^n
		},
	\end{align}
	which can be done via mixed integer quadratic optimization solvers, which usually employ a, potentially expensive, spatial B\&B strategy. However, this approach does not take into account any information on the objective function $f$. In addition, one needs to solve a mixed integer quadratic problem, which can be challenging even for well-developed solvers such as Gurobi. In the following section, we propose an alternative procedure based on concave tents where data on the objective function is utilized and calculations are cheaper at least in theory.  
	
	\begin{rem}
		Based on the discussion in \cite{burer_copositive_2009}, there is an immediate strengthening of $\GG_{\XX}$ where nonnegative slack variables for the redundant constraints $\e\geq \x \geq -\e$ are introduced and lifted in a similar manner as the slack variables $\s$ discussed above. This would increase the size of the psd block by $2n$, and parts of the lifted blocks can be subjected to nonnegativity constraints. In our experiments, we found that the computational burden of this construction would outweigh the benefit of obtaining tighter bounds substantially. This is not surprising, as modern conic solvers scale relatively poorly with the size of the psd block and often struggle if matrix variables are subject to both psd constraints and nonnegativity constraints at the same time. For this reason, we do not discuss this approach here any further. However, more powerful conic solvers might be able to tip the scale in favor of the more involved construction, which would be an interesting field of future investigations. 
	\end{rem}

	\begin{rem}
		In an earlier attempt to apply our methods to discrete robust optimization problems, we considered network flow problems with uncertain objective and ellipsoidal uncertainty sets, which are NP-hard even if the nominal problem is not (see \cite{bertsimas_robust_2003,bertsimas_robust_2004}). A reviewer from an earlier submission also pointed this out as an interesting use case. However, when we constructed the respective instance of $\GG_\XX$ (in a similar manner as described above) and calculated the respective lower bound and $\x$ we often found that it was extremely close to being binary, different rounding schemes always returned identical solutions and the branch and bound scheme terminated at or close to the root node. In contrast, the present application reliably exhibits poor performance of the lower bounding procedure, so that primal heuristics were much more impactful. Also, the two different rounding schemes have substantial effects on the performance of the branch and bound procedure. Hence, we decided to present only these results here. 
	\end{rem}
	
	\subsection{Heuristic upper bounds via concave tents }\label{sec:Heuristic} 
	Clearly, the function $f$ with ball uncertainty fulfills \cref{asm:CompactU} and $\XX$ is a set of extreme points of $\conv(\XX)$, so that by \cref{thm:ConcaveTentOverExtreme} we can derive a concave tent of $f$ over $\XX$ as follows 
	
	\begin{cor}\label{cor:CToverBoxQP}
		For $f$ and $\XX$ as defined in \eqref{eqn:Robust01QP} the following function is a concave tent that has an exact dual characterization and is finite for any $\x$ that is part of a feasible solution to \eqref{eqn:RobustSemiRelaxation}. 
		\begin{align*}
			\hspace{-0.8cm}
			g(\x)
			\coloneqq &
			\sup_{\Psi,\u,\Ub,\Xb}
			\lrbr{		
				\begin{array}{l}
					\Ab\bullet\Xb+2\a\T\x+2\Bb\T\bullet\Psi+\Cb\bullet\Ub+2\cc\T\u 
					\colon \\
					\begin{bmatrix}
						1  & \u\T & \x\T \\
						\u & \Ub  & \Psi\T \\
						\x & \Psi & \Xb
					\end{bmatrix} \in \SS^{n+q+1}_+,\
					\begin{array}{ll}
						(\x,\Xb)\in\GG_{\XX},& \hspace{-1.8cm}\| u\u- \Psi\T\e\|_2\leq u-\e\T\x,\\ 
						(\u,\Ub)\in\GG(\UU), & \hspace{-1.8cm}\| \hspace{0.05cm} l \hspace{0.05cm} \u- \Psi\T\e\|_2\leq \e\T\x-l,\\
						\|\u+\Psi\T\e_i\|_2 \leq 1+\e_i\T\x, & \forall  i\in\irg{1}{n}, \\
					\end{array}				
				\end{array}			
			}.
		\end{align*}
	\end{cor}
	\begin{proof}
		We need to show that the requirements of \cref{thm:ConcaveTentOverExtreme} are met. Firstly, note that the constraint $\diag(\Xb) = \e$, hidden in $(\x,\Xb)\in\GG_{\XX}$, takes the role of $\AA(\Xb,\x) = \b$. We see that for an extreme point of $(\x,\Xb)\in\GG(\XX)$ we have $\diag(\Xb) = \diag(\x\x\T) = \x\circ\x = \e$ since $\x\in\lrbr{-1,1}^n$ and $\diag(\Xb) = \oo$ implies $\Xb = \Ob$ whenever $\Xb\in\SS_+^{n}$. Thus, \eqref{eqn:Nec1} and \eqref{eqn:Nec2} hold. The sum total of the constraints constitutes our chosen instances of $\CC$. The positive semidefiniteness constraint guarantees that $\CC\subseteq \SS_+^{n+q+1}$ as required. We need to show that the other constraints are valid on $\GG(\UU\times\XX)$. Since $\GG_{\XX}\supseteq\GG(\XX)$ the respective constraint as well as the constraint $(\u,\Ub)\in\GG(\UU)$ are valid. Thus, to show that $\CC$ engulfs $\GG(\UU\times\XX)$ we need to demonstrate that the additional second-order cone constraints are valid as well. These constraints have been heavily inspired by the perspectification-linearization techniques recently developed in \cite{noauthor_zhen_nodate}. We start by observing that the constraints $\|\u\|_2\leq 1$ and $u-\e\T\x\geq 0$ are valid on $\UU$ and $\XX$ respectively. We can multiply the left-hand side of the latter with both sides of the former to observe the following equivalence
		\begin{align*}
			(u-\e\T\x)\|\u\|_2\leq (u-\e\T\x),
			\Leftrightarrow&
			\|u\u-\u\x\T\e\|_2\leq (u-\e\T\x),
		\end{align*} 
		so that the latter inequality yields another constraint valid on $\UU\times\XX$. This step is followed by a linearization step where $\x\u\T$ is replaced by $\Psi\in\R^{n\times q}$. This shows that our constraint is indeed valid on $\GG(\UU\times\XX)$ as it is valid on the extreme points of that set since for these we have $\Psi = \x\u\T$ with $[\u\T,\x\T]\T\in\UU\times\XX$. The second soc (second order cone) constraint is constructed analogously where we multiply the valid constraints $\|\u\|_2\leq 1$ and $\e\T\x-l\geq 0$. Also, the final set of constraints stems from an analogous construction where we multiplied again the constraint $\|\u\|_2\leq 1$ with each of the $n$ constraints $1+\e_i\T\x\geq 0, \ i \in\irg{1}{n}$ which are valid on $\XX$ since $(\x)_i \geq -1, \ i \in\irg{1}{n}$ due to the binarity constraints.   
		Finally, this primal characterization of $g$ has a compact feasible set as \eqref{eqn:Nec2} and \cref{asm:CompactU} hold so that the duality gap vanishes. For any $\x$ that is part of feasible solution of \eqref{eqn:RobustSemiRelaxation} there is an $\Xb\in\SS^n$ such that $(\x,\Xb)\in\GG_{\XX}\subseteq\SS_+^{n+1}$ so that  $(\x,\Xb)= \sum_{i=1}^k\gl_i(\x_i,\x_i\x_i\T)$ for some $\ggl\in\Delta_k$ and $\x_i\in\R^n, \ i \in\irg{1}{k}$. We also get $u\geq\e\T\x\geq l$ from the linear constraints in $\GG_{\XX}$ and $\x\in[0,1]^n$ is implied by the psd constraints and $\diag(\Xb) = \e$. Pick $\u\in\UU$ and set $\Ub = \u\u\T$ and $\Psi = \x\u\T$, so that $(\u,\Ub)\in\GG(\UU)$ . We see that the
		\begin{align*}
			\hspace{-0.5cm}
			\begin{bmatrix}
				1  & \u\T & \x\T \\
				\u & \Ub  & \Psi\T \\
				\x & \Psi & \Xb
			\end{bmatrix}
			=
			\begin{bmatrix}
				1  & \u\T & \sum_{i=1}^{k}\gl_i\x_i\T \\
				\u & \u\u\T  & \sum_{i=1}^{k}\gl_i\u\x_i\T \\
				\sum_{i=1}^{k}\gl_i\x_i & \sum_{i=1}^{k}\gl_i\x_i\u\T & \sum_{i=1}^{k}\gl_i\x_i\x_i\T
			\end{bmatrix}
			= \sum_{i=1}^{k}
			\gl_i
			\begin{bmatrix}
				1  & \u\T & \x_i\T \\
				\u & \u\u\T  & \u\x_i\T \\
				\x_i & \x_i\u\T & \x_i\x_i\T
			\end{bmatrix}\in\SS_+^{n+q+1}.\
		\end{align*}
		We also have  
		\begin{align*}
			\| u\u- \Psi\T\e\|_2 = \| u\u- (\sum_{i=1}^{k}\gl_i\x_i\u\T)\T\e\|_2 = \|\u(u-\e\T\x)\|_2 = (u-\e\T\x)\|\u\|_2\leq u-\e\T\x,
		\end{align*}   
		and the other soc constraints are validated in a similar manner so that $(\Psi,\u,\Ub,\Xb)$ is a feasible solution for $g(\x)$ which is thus finite. 
	\end{proof}
	
	\begin{rem}
		We want to point out that the soc constraints that strengthen $\CC$ in the concave tent described above cannot be used to strengthen the psd-relaxation \eqref{eqn:RobustSemiRelaxation} since $\Psi$ does not appear there. This demonstrates that our theory makes the techniques developed in \cite{noauthor_zhen_nodate} more broadly applicable.
	\end{rem}
	
	By \cref{thm:EpsilonSuperDerivatives}, any $\eps$-optimal feasible solution for the dual yields a vector $\y_{\eps}\in\R^n$ that is an $\eps$-supergradient of $g$, where $\eps$ can be made arbitrarily small. This supergradient can be used to construct a feasible solution, say $\x_{ub}$, to \eqref{eqn:Robust01QP} by solving 
	$
	\min_{\x\in\R^n}\lrbr{\y_{\eps}\T\x\colon u\geq \e\T\x\geq l, \ 
		\x\in\lrbr{-1,1}^n},
	$ 
	which can be solved using MILP solvers such as Gurobi. Based on this idea we propose the following procedure that is a problem-data-driven alternative to the rounding in \eqref{eqn:RobustClosestFeasible}: 
	
	\begin{algorithm}[H]
		\SetAlgoLined
		\KwResult{$\x_{ub}, \ f(\x_{ub})$}
		input $\eps>0$
		and  $\x_{rel}$ from an optimal solution of \eqref{eqn:RobustSemiRelaxation},\;
		find $\eps$-optimal solution dual characterization of $g(\x_{rel})$, to obtain $\y_{\eps}$, \;
		obtain $\x_{ub} \in \arg\min_{\x\in\R^n}\lrbr{\y_{\eps}\T\x\colon u\geq \e\T\x\geq l, \ 
			\x\in\lrbr{-1,1}^n}$, \;
		calculate $f(\x_{ub})$ from \eqref{eqn:fisanSDP}.
		\caption{\footnotesize{Primal heuristic.}
			\label{alg:PrimalHeuristic}
		}
	\end{algorithm}
	
	\begin{rem}
		Since $\epsilon$ can be made arbitrarily small, we never used it as an input in our numerical experiments which we will describe in \cref{sec:Numerical}. Rather, we solved the conic problem to the optimality tolerance of the conic solver to make $\epsilon$ as small as possible, which is desirable since a smaller $\epsilon$ corresponds to a more faithful model of the concave tent.  
	\end{rem}
	
	This procedure is designed in the spirit of \cref{thm:ConcaveReformulations} a), as we try to approximate local solutions to our problem by minimizing the concave tent locally and approximately by minimizing its linear model at $\x_{rel}$ over $\XX$. One could iterate these steps in a Frank-Wolfe fashion, but this would be computationally prohibitive. Note, that linear optimization over $\XX$ is equivalent to linear optimization over $\conv(\XX)$, and when solving the mixed integer linear problem in \cref{alg:PrimalHeuristic}, we typically exploit valid cuts for this convex hull. Thus, the algorithm exemplifies that concave tents allow us to use knowledge about $\conv(\XX)$ to derive upper bounds for nonconvex optimization problems over $\XX$. 
	
	We also see that in contrast to \eqref{eqn:RobustClosestFeasible} we only need to solve a conic problem and then a MILP, both of which are much more easily handled than an MIQP. In addition, the concave tent exploits information about the objective, while such information is absent from \eqref{eqn:RobustClosestFeasible}. However, since the conic optimization problem involves psd-constraints the computational burden can still be substantial. Before proceeding we will illustrate our constructions in this the and previous section in a small example.  
	
	\begin{exmp}
		For this example, we consider the case $n=2,\ q=2,\ u = 1, \ l = 0$ in which case $\XX = \lrbr{[1,-1]\T,[-1,1]\T}$ so that its convex hull is the respective line segment. We created a random instance of $f$, but we omitted to give the precise values for the data as this is not important for the example. What is important is that we constructed two different concave tents: one that follows the construction in \cref{cor:CToverBoxQP} and a second one that omits the additional second-order cone constraints in that construction. In \cref{fig:GoodvsBad} we depict $f$ in purple and the two concave tents and red and green both from the side and the top. The smaller concave tent is the one with the additional constraints in $\CC$, depicted in the right column of the figure. The convex hull of $\XX$ is the line segment in red. 		
		\begin{figure}[h]
			\centering{
				\begin{tabular}{cc}				
					\includegraphics[scale=0.35]{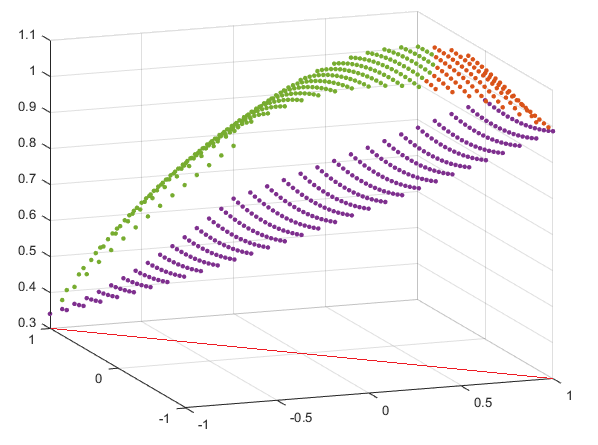} & 
					\includegraphics[scale=0.35]{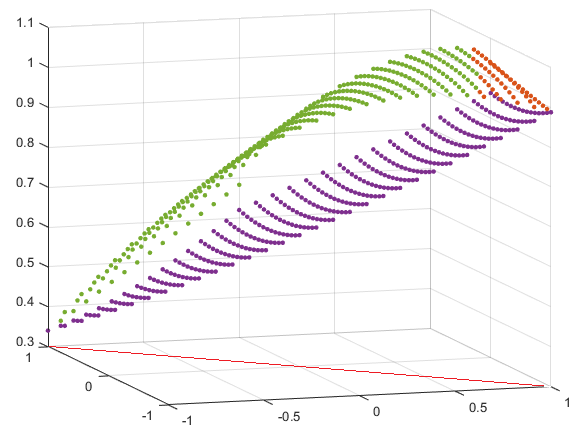}\\
					\includegraphics[scale=0.35]{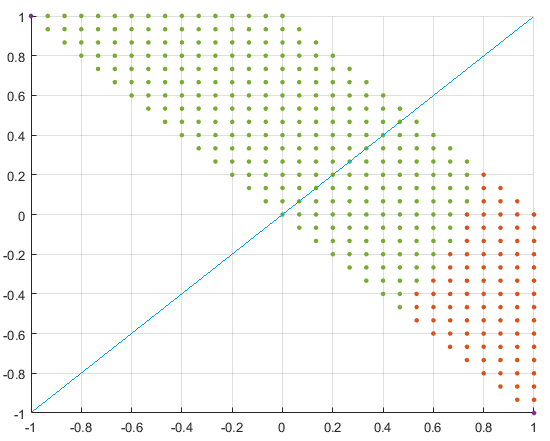} & 
					\includegraphics[scale=0.35]{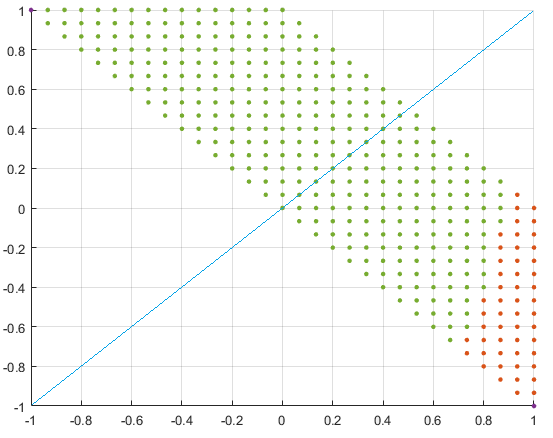}
				\end{tabular}
			}				
			\caption{Comparing concave tents}
			\label{fig:GoodvsBad}		
		\end{figure}
		
		Firstly, we see that both concave tents agree with $f$ on $\XX$, i.e.\ the two endpoints of the line segment. Secondly, we also see that both are finite on $\XX_{rel}\coloneqq\lrbr{\x\in[-1,1]^2\colon 1\geq \e\T\x\geq 0}$ since $\GG_{\XX}$ is not strong enough to cut off spurious parts of the effective domain of the concave tents. The most interesting part is the coloration of the two concave tents. We started \cref{alg:PrimalHeuristic} from different grid-points in $\XX_{rel}$ instead of starting it from an optimal solution to \eqref{eqn:RobustSemiRelaxation} and checked whether the point obtained from this procedure is the global minimizer of $f$ over $\XX$, i.e. the leftmost point on the line segment. If so, the respective point was colored green, otherwise it was colored red. We see that the concave tent with the additional constraints is more likely to guide our heuristic toward the minimizer. In the top-view images, we have also added a blue line that separates the points that are closer to the left-most point, i.e. the minimizer, from those that are closer to the rightmost point. We see that the classical rounding based on \eqref{eqn:RobustClosestFeasible} would have rounded many points to the undesirable rightmost point. 		
		
		Of course, this behavior is not guaranteed, and it depends on the geometry of both $\XX$ and $f$. The only thing we can guarantee for the smaller concave tent is that it does not have more local minimizers than the larger one, which is of course beneficial for our primal heuristic. In our experiments discussed in the next section, the numerical benefits of including the second-order cone constraints were substantial, which is why these were used exclusively.  		
	\end{exmp}
	
	\section{Numerical evidence in a B\&B framework} \label{sec:Numerical}
	In this section, we will present results from numerical experiments where we tested the procedure described in the previous section. The experiments were carried out on a laptop with an Intel Core i5-9300H CPU with 2.40GHz and 16 GB of RAM. As a modeling interface we used YALMIP \cite{lofberg_yalmip_2004}, as a conic solver we used Mosek 9.2 and as a MILP- and MIQP-solver we used Gurobi 9.1. We made heavy use of Mosek's ability to report optimal dual variables from primal formulations to
	avoid working with the duals explicitly. This reduced the complexity of the implementation but
	also aided numerical stability in our experience. 
	
	In our experiments, we solve \eqref{eqn:Robust01QP} via a B\&B procedure which we implemented ourselves. For the lower bounds, we use \eqref{eqn:RobustSemiRelaxation}, which typically yields fractional solutions. The idea is to 'round' these via our primal heuristic and our experiment aims at finding an advantage of this procedure compared to the traditional rounding described in \eqref{eqn:RobustClosestFeasible}. The question is whether there is a difference in the number of nodes explored and, as a possible consequence, in the overall running time of the solver. Better upper bounds can lead to earlier pruning so that at least the former is plausible, while the latter is additionally affected by the computational burden of evaluating the concave tent and solving the MILP in \cref{alg:PrimalHeuristic} versus solving the MIQP in \eqref{eqn:RobustClosestFeasible}.  
	
	As a branching strategy we used a "last-in-first-out" principle, where after branching and solving the two new subproblems, we choose the leave with the better lower bound to enter the list of open nodes last. Whenever a subproblem was solved, the so obtained fractional solution was used to construct a feasible integer solution either via our primal heuristic or via traditional rounding.  
	
	We generated instances in the following manner: let $\phi_{\MM}$ be the uniform distribution over a set $\MM$. For the matrices $\Ab\in\SS^n$ and $\Cb\in\SS^q$ we generated matrices $\tilde{\Ab}\in\R^{n\times n}, \ (\tilde{\Ab})_{ij}\sim \phi_{[-0.5,0.5]}$ and $\tilde{\Cb}\in\R^{q\times q}, \ (\tilde{\Cb})_{ij}\sim \phi_{[-0.5,0.5]}$ and set $\Ab = (\tilde{\Ab}+\tilde{\Ab}\T)/n^2, \ \Cb = (\tilde{\Cb}+\tilde{\Cb}\T)/q^2$. We set the matrix $\Bb = \tilde{\Bb}/(qn)\in\R^{q\times n}$ where $(\tilde{\Bb})_{ij} \sim \phi_{[0,1]}$. For the vectors we generate $\a\in\R^n$ and $\cc\in\R^q$ such that $(\a)_i \sim \phi_{[-0.5,0.5]/n^2}$ and $(\cc)_i \sim \phi_{[0,1]/q^2}$
	
	The results of our experiments are summarized in \cref{tbl:Results}. We considered various instance types with different values for $n$, $q$, $l$, and $u$. For each of those types, we generated $5$ instances at random as described above. We solved these instances via our B\&B implementation where the upper bounds were either generated via classical rounding (results under "\textbf{R}") or with our primal heuristic (results under "\textbf{CT}"). We compare the performance based on the number of nodes investigated before convergence and on the total solution time. The table represents only a subset of instances we tested and we chose to include a range of instances that summarizes all the phenomena that we have encountered in our experiments.

	\begin{table}[h]
		\centering{
			\begin{tabular}{r||rr|rr}	
				\multicolumn{1}{c||}{\textbf{Instances}}& \multicolumn{ 2}{c|}{\textbf{Node count}} & \multicolumn{ 2}{c}{\textbf{Solution time }} \\ 
				\multicolumn{1}{c||}{$n\_q\_l\_u$}& \multicolumn{1}{c}{\textbf{R}} & \multicolumn{1}{c|}{\textbf{CT}} & \multicolumn{1}{c}{\textbf{R} } & \multicolumn{1}{c}{\textbf{CT}} \\ \hline
				\textbf{30\_10\_-5\_5:} 1 & 365 & 369 & \textbf{245,32} & 298,34 \\ 
				2 & \textbf{409} & 413 & \textbf{279,46} & 330,89 \\ 
				3 & \textbf{195} & 289 & \textbf{123,38} & 234,85 \\ 
				4 & 1045 & 1045 & 1069,50 & 925,71 \\ 
				5 & \textbf{1053} & 1135 & \textbf{962,93} & 1095,39 \\ \hline
				\textbf{30\_20\_10\_15:} 1 & 77 & \textbf{75} & \textbf{46,90} & 81,67 \\ 
				2 & 207 & 207 & \textbf{131,18} & 211,24 \\ 
				3 & 243 & \textbf{159} & \textbf{156,37} & 161,93 \\ 
				4 & 71 & 71 & \textbf{43,32} & 77,63 \\ 
				5 & 47 & 47 & \textbf{28,40} & 53,30 \\ \hline					   
				\textbf{30\_10\_12\_15:} 1 & 77 & 77 & \textbf{45,78} & 63,29 \\ 
				2 & 101 & 101 & \textbf{60,82} & 83,68 \\ 
				3 & 55 & \textbf{51} & \textbf{32,26} & 42,57 \\ 
				4 & 89 & \textbf{35} & 52,34 & \textbf{29,35} \\ 
				5 & 197 & \textbf{149} & 119,82 & \textbf{118,98} \\ \hline		
				\textbf{30\_5\_15\_20:} 1 & 1039 & 1039 & 874,54 & \textbf{692,62} \\ 
				2 & 431 & 431 & \textbf{288,26} & 295,92 \\ 
				3 & 1795 & 1795 & 1885,06 & \textbf{1181,89} \\ 
				4 & 3037 & \textbf{2529} & 4140,64 & \textbf{1643,06} \\ 
				5 & 1439 & 1439 & 1468,53 & \textbf{987,53} \\ \hline
				\textbf{30\_5\_5\_20:} 1 & 433 & 433 & \textbf{246,13} & 416,25 \\ 
				2 & 1877 & 1877 & 2102,73 & \textbf{1301,26} \\ 
				3 & 1087 & \textbf{701} & 969,04 & \textbf{497,52} \\ 
				4 & 4445 & \textbf{4441} & 7983,53 & \textbf{2974,38} \\ 
				5 & 1911 & 1911 & 2152,27 & \textbf{1311,57} \\ 
			\end{tabular}		
		}		
		\caption{Results of the experiments}
		\label{tbl:Results}	
	\end{table}
	In the first group of instances, the classic rounding approach mostly outperforms our approach in terms of both measures. We suspect that this is because the objective function seems to be relatively symmetrical around the origin so that the concave tents were relatively flat over $\conv(\XX)$. In effect, our Frank-Wolfe step was not propelled in beneficial directions. In the second group, the situation is more even in terms of nodes explored with the concave tent approach doing a little better, but the computation time mostly favored the classical approach even in instances where the node count favored the concave tent approach. In the third group node count favors the concave tents while computation time is mixed.  The results of both of these groups demonstrate that, even though the rounding procedure is based on an MIQP, solving the SDP for the concave tent approach can inflict a higher computational burden. Note, that the larger $q$ the bigger the matrix block in the concave tent computation, which shows in the overall computation time as well, which favored our primal heuristic for instances where $q$ was small. For example, in the fourth group computation time favored the concave tent approach while node counts were similar except for instance 4. This behavior continues in the fifth group where our heuristic performs well in both measures. We suspect that the looser linear constraints in the last group made it difficult for the Gurobi to solve the MIQPs since the linear relaxations it employs for globally solving it are weakened substantially. On the other hand, Gurobi's MILP solver and Mosek's barrier SDP solver were not affected by this as much in terms of computation time. Overall we can see that our approach can be beneficial depending on the problem at hand and in our estimations merits further investigation in future research. 
	
	\subsection*{Conclusion}
	\vspace{-0.5cm}
	In this article, we introduced concave tents and studied some theoretical properties as well as their practical construction for a large class of use cases. We investigated some useful aspects of these constructions and used these insights to develop a rounding heuristic for a certain type of robust optimization problem. When compared to a traditional rounding scheme this approach performs promisingly in a B\&B framework. We hope we can expand on this theory in the future.    
	\vspace{-0.6cm}
	\paragraph{Acknowledgments:} The research of M.G. is financially supported by the FWF project ESP 486-N. 
	\vspace{-0.5cm}
	\setlength{\bibsep}{0.5pt}
	\bibliography{Literature} 
	\bibliographystyle{abbrv}
	\appendix
	\vspace{-0.3cm}
	\section{Regularizing convex functions}\label{apx:RegularizingConveFunctions}
	As stated in \cref{sec:A versatile class of functions}, the bidual characterization of a closed, proper convex function, say $c$, almost fulfills \cref{asm:CompactU}, but the resulting instance of $\UU$ may lack the critical property of compactness. We will therefore switch to a regularized version of $c$ given by its so-called Pasch-Hausdorff envelope: 
	\begin{align*}
		c_{\gb}(\x) 
		\coloneqq 
		\inf_{\y\in\R^n} 
		\lrbr{
			c(\y)-\gb\|\x-\y\|_2
		} 
		= 
		\sup_{\u\in\R^n}\lrbr{\x\T\u-c_{\gb}^*(\u)}	
		= 
		\sup_{\u\in\R^n}\lrbr{\x\T\u-c^*(\u)\colon \|\u\|_2\leq \beta},
	\end{align*}  
	where the definition follows \cite[Section 12.3]{bauschke_infimal_2017} and the equations are due to \cite[Lemma 2.1]{beck_exact_2024}. Since $\dom c^*\cap \lrbr{\u\in\R^n\colon \|\u\|_2\leq \gb}$ is compact (as closed convex functions have closed effective domains) there is a $\hat{u}$ such that 
	\begin{align*}
		c_{\gb}(\x)  = \sup_{\u,u_0}\lrbr{\x\T\u-u_0\colon  
			\colon 
			\left[\u\T,u_0\right]
			\in\UU
		}, \quad 
		\UU
		\coloneqq
		\lrbr{
			\begin{bmatrix}
				\u \\ u_0
			\end{bmatrix}\in\R^{n+1}
			\colon 
			\begin{array}{l}
				c^*(\u)\leq u_0\leq \hat{u},\\
				\|\u\|_2\leq \beta,
			\end{array}			
		},				
	\end{align*}
	so that $c_{\gb}$ fulfills all the requirements of \cref{asm:CompactU}, since $c^*$ is closed and proper. While $c_{\gb}$ is a minorant in $c$ in general the following result states that for any compact $\XX\subseteq \dom c$  we can find $\gb$ such that $c_{\gb}(\x) = c(\x), \ \forall \x \in \XX$.
	\begin{thm}
		Let $c$ be a closed proper convex function and let $\XX$ be a compact subset of $\dom c$. Then there is a $\gb>0$ such that $c_{\gb} = c$ on $\XX$. 
	\end{thm}
	\begin{proof}
		This is a direct consequence of \cite[Corollary 12.19]{bauschke_infimal_2017}, after realizing that all closed proper convex functions are Lipschitz continuous relative to any compact subset of their effective domain. 
	\end{proof}  
	The downside of this derivation is that it requires calculating an upper estimate $\gb$ for the Lipshitz constant. This can be difficult in general and a poor estimation might lead to weaker concave tents. We will defer the investigation of cases where this can be done sufficiently accurately to future research. 
	
\end{document}